\documentclass[11pt,reqno]{amsart}%
\usepackage[a4paper,hscale=0.7,vscale=0.75,centering]{geometry}
\usepackage{amssymb}
\usepackage{color}
\usepackage{amsmath}
\usepackage{amsfonts}
\usepackage{graphicx}%
\usepackage{ulem}

\usepackage{dsfont,mathrsfs,stmaryrd,wasysym,amsbsy}

\newcommand {\E}{{\mathrm E}}

\newcommand{\R}{\mathbb{R}}
\newcommand{\B}{\mathcal{B}}
\newcommand{\F}{\mathcal{F}}

\newcommand\SO[1]{{\mathrm{SO}}_{#1}}
\newcommand\SOR[1]{{\mathrm{SO}}_{#1}(\R)}
\newcommand\Sph{{\mathbb S}}

\renewcommand{\P}{\mathrm{P}}

\newtheorem{prop}{Proposition}[section]
\newtheorem{thm}[prop]{Theorem}
\newtheorem{cor}[prop]{Corollary}
\newtheorem{lemma}[prop]{Lemma}
\newtheorem{defi}[prop]{Definition}

\theoremstyle{definition}

\theoremstyle{remark}
\newtheorem{rmk}[prop]{Remark}
\numberwithin{equation}{section}

\def \B{\bar{B}}

\def \N{\mathbb N}

\def \I{\mathbb I}
\def\T{\mathbb T}

\begin{document}

\title[The Landau Equation and 
the Brownian Motion on $\SOR{N}$.]{The Landau Equation for Maxwellian molecules
and 
the Brownian Motion on $\SOR{N}$. }

\author{Fran\c{c}ois Delarue}
\address[Fran\c{c}ois Delarue]{Laboratoire J.A.Dieudonn\'e, UMR 7351,
Universit\'e Nice Sophia-Antipolis,
Parc Valrose, 06108 Nice, France}
\email{delarue@unice.fr}

\author{St\'ephane Menozzi}
\address[St\'ephane Menozzi]{Laboratoire de Mod\'elisation Math\'ematique d'Evry, UMR 8071,
Universit\'e d'Evry Val d'Essonne, 23 Boulevard de France, 91037 Evry, France}
\email{stephane.menozzi@univ-evry.fr}

\author{Eulalia Nualart}
\address[Eulalia Nualart]{Department of Economics and Business, Universitat Pompeu Fabra and Barcelona Graduate School of Economics, Ram\'on Trias Fargas 25-27, 08005
Barcelona, Spain}
\email{eulalia@nualart.es}
\urladdr{http://nualart.es}

\thanks{Part of the work for this paper was done when the authors where visting the 
Centre Interfacultaire Bernoulli of the \'Ecole Polytechnique F\'ed\'erale de Lausanne, during the semester on Stochastic Analysis and Applications organized by Robert Dalang, 
Marco Dozzi, Franco Flandoli and Francesco Russo. We are very grateful for their invitation and hospitality. Third author acknowledges support from the European Union programme
 FP7-PEOPLE-2012-CIG under grant agreement 333938}

\begin{abstract}
In this paper we prove that the 
spatially homogeneous Landau equation for Maxwellian molecules
can be represented through the product of two elementary processes. The first one is 
the Brownian motion on the group of rotations. The second one is, 
conditionally on the first one, a Gaussian process. Using this representation, we establish sharp multi-scale upper and lower 
bounds for the transition density of the Landau equation, the 
multi-scale structure depending on the shape of the support of the initial condition.  
\end{abstract}
\date{June 2014}
\maketitle

\section{Statement of the problem and existing results}
The spatially homogeneous Landau equation for Maxwellian molecules is a common model in plasma physics. It can be obtained as a certain limit of the spatially homogeneous Boltzmann equation for 
$N$ dimensional
particles 
subject to pairwise interaction,
 when the collisions become grazing
 and when the interaction forces 
 between particles at distance $r$
 are
order $1/r^{2N+1}$
 (see Villani \cite{Villani:98b}
and Gu\'erin \cite{Guerin:04}). 

The Landau equation reads as a
nonlocal Fokker-Planck equation. Given an initial condition $(f(0,v),v \in \R^N)$, the solution is denoted by 
$(f(t,v), t \geq 0, v \in \R^N)$, $N \geq 2$, and satisfies
\begin{equation} \label{1}
\partial_t f(t,v)=Lf(t,v),
\end{equation}
where
\begin{equation} \label{2}
Lf(t,v) =\nabla \cdot \int_{\R^N} dv_{\ast} \, a(v-v_{\ast}) \left( f(t, v_{\ast}) 
\nabla f(t,v)-f(t,v) \nabla f(t, v_{\ast}) \right).
\end{equation}
Here, $a$ is an $N \times N$ 
nonnegative and symmetric matrix that depends on the collisions between binary particles. It is given by
(up to a multiplicative constant)
$$
a(v)=\vert v \vert^ 2 \text{Id}_N -v \otimes v,
$$
where Id$_N$ denotes the  identity matrix of size $N$, and $v \otimes v=  vv^{\top}$, $v^{\top}$ denoting the transpose of $v$, $v$ being seen as a column vector in $\R^N$.
The unknown function $f(t,v)$ represents the density of particles of velocity $v \in \R^N$ 
 at time $t \geq 0$ in a gas. It is assumed to be independent of
the position of the particles (spatially homogeneous case).

The density $f(t,v)$ being given, the nonlocal operator 
$L$ can be seen as a standard linear Fokker-Planck operator,
with diffusion matrix $\overline{a}(t,v)
=\int_{\R^N}  a(v-v_{\ast}) f(t, v_{\ast}) dv_{\ast}$ and with 
drift $
\overline{b}(t,v)= - (N-1) \int_{\R^N}  (v-v_{\ast}) f(t, v_{\ast}) dv_{\ast}$.
Such a reformulation permits to approach the Landau equation by means of the
numerious
 tools that have been developed for linear diffusion operators.
 As a key fact in that direction, the diffusion matrix 
$\overline{a}$ can be shown to be uniformly elliptic
for a wide class of initial conditions. This suggests that
the 
solution $f(t,v)$ must share some of the generic properties of non-degenerate diffusion operators. 

Such a remark is the starting point of the analysis initiated by 
Villani in \cite[Proposition 4]{Villani:98}. Therein, it is proved
that, whenever the initial condition $f(0,v)$ is nonnegative and has finite mass and energy,
the Landau PDE (\ref{1}) admits a unique solution, which is bounded and 
$\mathcal{C}^{\infty}(\R^N)$ in positive time.
Moreover, \cite[Proposition 9]{Villani:98} ensures that the solution satisfies the lower Gaussian bound
\begin{equation} \label{co}
f(t,v )\geq C_t e^{-\delta_t \frac{\vert v \vert^2}{2}}, \ t>0, \ v \in \R^N, 
\end{equation}
for some $C_t>0$ and $\delta_t>0$.
The values of the constants $C_{t}$ and $\delta_{t}$ are specified in
Desvillettes and Villani \cite[Theorem 9(ii)]{DesVillani:00} when $N=3$,
under the additional condition that $f(0,v)$ has finite entropy and is bounded from 
below  
by a strictly positive constant on a given ball. The lower bound \eqref{co}
is then established with $C_{t}=1$ and $\delta_{t} = b_{0}t + c_{0}/t$. 
This proves that, in finite time, the rate of propagation of the
 mass to the
infinity is at least the same as for the heat equation. 
The key argument in \cite{DesVillani:00}
is to prove that the spectrum of $\bar{a}(t,v)$ is uniformly far away from zero, 
so that the mass can be indeed diffused to the whole space. 

Anyhow, even if the lower bound \eqref{co} fits 
the off-diagonal decay of the heat kernel, 
it is worth mentioning that $\bar{a}(t,v)$ does not enter the required framework 
for applying two-sided Aronson's estimates for diffusion operators, see \cite{aron:67}. Indeed, the upper eigenvalue
of $\bar{a}(t,v)$  
can be shown to behave as $\vert v \vert^2$
when $\vert v \vert$ is large. The matrix $\bar{a}(t,v)$ thus exhibits several scales when $\vert v \vert$ tends to the infinity, 
which is the
basic observation for motivating our analysis. Actually,
a simple inspection will show that, for the same type of initial conditions as above, the
quadratic form associated with $\bar{a}(t,v)$ has two regimes when $\vert v \vert$ is large. 
Along unitary vectors parallel to $v$, the quadratic form takes values of order $1$. Along 
unitary vectors 
orthogonal to $v$, it takes values of order $\vert v \vert^2$. 
This suggests that the mass is spread out at a standard diffusive rate along \textit{radial} directions, 
but at a much quicker rate along \textit{tangential} directions. 
One of the main objective of the paper is to quantify this phenomenon precisely and 
to specify how it affects 
the lower bound 
\eqref{co}, especially for highly anisotropic initial conditions. We also intend to discuss the sharpness of the bound by 
investigating the corresponding upper bound. 

The strategy we have in mind is probabilistic. The starting point consists in deriving a probabilistic 
interpretation of the 
nonlinear operator $L$ by means of a stochastic diffusion process
$(X_{t})_{t \geq 0}$ interacting with its own distribution, 
in the spirit of McKean to handle Vlasov type equations (see Sznitman \cite{sznitman}). 
Actually, McKean-Vlasov representations of the Landau equation were already investigated in 
earlier works by Funaki 
\cite{funaki:publiRIMS:83,funaki:ZW:84,funaki:duke:85,
funaki:LNM}
and more recently by Gu\'erin
\cite{Guerin:02, Guerin:03}. Part of the analysis developed in these series of papers is based on 
a very useful trick for representing the square root of the matrix $\overline{a}$, 
the square root of the diffusion matrix playing a key role in the dynamics of the stochastic process involved in 
the representation. In short, the key point therein is to enlarge 
the underlying probability space in order to identify the diffusive term 
with the stochastic integral of the root of $a$ (and not the root of 
$\overline{a}$) with respect to a two-parameter white noise process. 
Basing the representation on the root of $a$ makes it more tractable since 
$a(v)$ has a very simple geometric interpretation in terms of the orthogonal 
projection on the orthogonal $v^{\perp}$ of $v$. 
In this paper, we go one step forward into the explicitness of the representation. 
As a new feature, we show that the representation  
used by Funaki and Gu\'erin can be linearized so that the 
stochastic process $(X_{t})_{t \geq 0}$ solving the \textit{enlarged} McKean-Vlasov
equation reads as the product of 
two \textit{auxiliary} basic processes:
\begin{equation*}
X_{t} = Z_{t} \Gamma_{t}, \quad t \geq 0. 
\end{equation*}
The first one is a (right) Brownian motion $(Z_{t})_{t \geq 0}$
on the special group of rotations $\SOR{N}$. The second one is, conditionally on $(Z_{t})_{t \geq 0}$, a Gaussian process in $\R^N$ with a local covariance matrix 
given, at any time $t \geq 0$, 
by the second order moments of the density $f(t,v)$. 
Such a decomposition enlightens explicitly the coexistence of two scales in the dynamics of 
the Landau equation. It is indeed well seen that the Brownian motion on $\SOR{N}$ cannot play any role
in the diffusion of the mass along radial directions. Therefore, along such directions, 
only $(\Gamma_{t})_{t \geq 0}$ can have an impact. Its covariance matrix can be proved
to be uniformly non-degenerate for a wide class of initial conditions, 
explaining why, in such cases, 
the mass is transported along radial directions according to the standard heat propagation.  
The picture is different along tangential directions since, in addition to the 
fluctuations of $(\Gamma_{t})_{t \geq 0}$, the process $(X_{t})_{t \geq 0}$ 
also feels the fluctuations 
of the Brownian motion $(Z_{t})_{t \geq 0}$
on $\SOR{N}$. 
The effect of $(Z_{t})_{t \geq 0}$ is all 
the more visible when the process $(X_{t})_{t \geq 0}$ is far away from the origin: Because 
of the product form of the representation, 
the fluctuations in the dynamics of  $(Z_{t})_{t \geq 0}$ 
translate into multiplied fluctuations in the dynamics of $(X_{t})_{t \geq 0}$
when $(X_{t})_{t \geq 0}$ is of large size.

Our main result in that direction is Theorem 
\ref{bds_EXPL_UE}, in which we provide two sided Gaussian bounds for the transition kernel 
of the process $(X_{t})_{t \geq 0}$ when the initial condition $X_{0}$ 
is a centered random variable with a support not included in a line. 
We then make appear the coexistence of two regimes
in the transition density by splitting the 
off-diagonal decay of the density into a \textit{radial cost} and 
a \textit{tangential cost}. We explicitly show that the variance of 
the tangential cost increases at a quadratic rate when the starting point in the transition density 
tends to the infinity. The resulting bounds are sharp, 
which proves that our approach captures
the behavior of the process in a correct way. 
The proof follows from our
factorization of the process $(X_{t})_{t \geq 0}$: Conditionally on 
the Brownian motion on $\SOR{N}$, $(X_{t})_{t \geq 0}$
is a Gaussian process with an explicit transition kernel. This gives
a conditional representation of the transition density of $(X_{t})_{t \geq 0}$ 
and this permits to reduce part of the work to the analysis of the 
heat kernel on the group $\SOR{N}$. As a by-product, this offers 
an alternative to a more systematic probabilistic method based on 
the Malliavin calculus, as considered for instance in 
Gu\'erin, M\'el\'eard and Nualart \cite{Guerin:06}. 

The conditional representation of the transition density of the process $(X_{t})_{t \geq 0}$ 
also permits to consider the so-called \textit{degenerate} case when the initial condition lies
in a line. In that case, another inspection will show that
the diffusion matrix $\bar{a}(t,v)$ degenerates as $t$ tends to $0$, 
the associated quadratic form converging to 0 with $t$ along the direction of the initial condition. 
Obviously, this adds another difficulty to the picture given above: Because of the degeneracy 
of the matrix $\bar{a}$, the mass cannot be transported along radial directions as
in standard heat propagation. In that framework, our representation provides 
a quite explicit description of the degeneracy rate of the system in small time.
Indeed,  
conditionally on the realization of the Brownian motion $(Z_{t})_{t \geq 0}$
 on $\SOR{N}$, the degeneracy is determined by the covariance matrix 
of the process $(\Gamma_{t})_{t \geq 0}$, the form of which is, contrary to the 
non-degenerate case, highly sensitive to the realization of $(Z_{t})_{t \geq 0}$.
The crux is thus that, in the degenerate regime, the Brownian motion on the group of
rotations also participates in the formation of the radial cost. 
Although quite exciting, this makes things rather intricate. In that direction, 
the thrust of our approach is to prove that large deviations of the process $(Z_{t})_{t \geq 0}$
play an essential role in the shape of the off-diagonal decay of the transition density. 
Precisely, because of that large deviations, we can show that, when the initial condition of 
the transition is restricted to compact sets, the off-diagonal decay of the transition density
is not Gaussian but is a mixture of an exponential and a Gaussian regimes,
see Theorem \ref{thm:bd:degenerate}. 
 
Besides the density estimates, we feel that our representation of the solution 
raises several questions and could serve as a basis for further investigations. Obviously, the first one concerns possible 
extensions to more general cases, when the coefficients include a hard or soft potential (so that molecules are no
more Maxwellian) or when
the solution of the Landau equation also depends on the position of the particle (and not only on its velocity). 
In the same spirit, we could also wonder about a possible adaptation of this approach 
to the Boltzmann equation itself. Finally, the representation might be also useful 
to compute the solution numerically, providing a new angle to tackle with the particle approach developed by Fontbona \textit{et al} \cite{font:guer:mele:09} and Carrapatoso \cite{carr:12} or Fournier \cite{four:09}. We leave all these questions to further prospects.

 The paper is organized as follows. Main results are detailed in Section \ref{se:main}. 
 In Section \ref{SEC_DENS}, we give some preliminary estimates concerning the Brownian motion on $\SOR{N}$. 
Section \ref{SEC_NON_DEG} is devoted to the analysis of the non-degenerate case and 
 Section \ref{SEC_DEG} to the degenerate case.

\section{Strategy and Main Results}
\label{se:main}

\subsection{Representation of the Landau equation}

The representation used in 
\cite{funaki:publiRIMS:83,funaki:ZW:84,funaki:duke:85,funaki:LNM,Guerin:02, Guerin:03}
is based on a probabilistic set-up, which consists of
\begin{enumerate}
\item
 a complete probability space $(\Omega, \mathcal{F}, \P)$, endowed with  
an $N$-dimensional space-time white noise $W = (W^1,...,W^N)$ 
with independent entries, each of them with covariance measure $ds d\alpha$ on $\R_+\times [0,1]$,  where $d\alpha$ denotes the
Lebesgue measure on $[0, 1]$;
\item a random vector $X_{0}$ with values in  $\R^N$, independent
of $W$, the augmented filtration generated by $W$ and $X_0$ being denoted by
 $(\mathcal{F}_t)_{t \geq 0}$;
\item  the auxiliary probability space $([0,1],\mathcal{B}([0, 1]), d\alpha)$;
\item the symbols $\E$, $\E_{\alpha}$ for denoting the expectations and
the symbols $\mathcal{L}, \mathcal{L}_{\alpha}$ for denoting 
the distributions of a random variable on $(\Omega,\mathcal{F},\P)$, $([0, 1],\mathcal{B}([0, 1]), d\alpha)$, respectively.
\end{enumerate}

A couple of processes $({\mathcal X},Y)$ on $(\Omega,\mathcal{F}, (\mathcal{F}_t)_{t \geq 0}, \P) \otimes ([0, 1],{\mathcal B}([0, 1]), d\alpha)$ 
is said to be a solution of the Landau SDE if ${\mathcal{L}({\mathcal X}) = \mathcal{L}_{\alpha}(Y)}$, and for all $t \geq 0$, the following equation holds
\begin{equation} \label{lan}
{\mathcal X}_t=X_0+\int_0^t \int_0^1 \sigma({\mathcal X}_s-Y_s(\alpha)) W(ds, d\alpha)  -(N-1)\int_0^t \int_0^1 ({\mathcal X}_s-Y_s(\alpha)) d\alpha ds,
\end{equation}
where $\sigma$ is an $N\times N$ matrix such that 
$\sigma \sigma^{\top}=a$, the symbol $\top $ standing from now on for the transposition. 
Roughly speaking, 
the connection with (\ref{1}) can be derived by computing:
\begin{equation*}
\E \biggl[ \biggl(  \int_0^1 \sigma(x-Y_s(\alpha)) W(ds, d\alpha) \biggr) 
\biggl( \int_0^1 \sigma(x-Y_s(\alpha)) W(ds, d\alpha) \biggr)^{\top}\biggr] \bigg\vert_{x = \chi_{s}}
= \bar{a}(s,\chi_{s}) ds, 
\end{equation*}
thus identifying the local covariance in \eqref{lan} with the diffusion matrix $\bar{a}$. Existence and uniqueness of a solution to \eqref{lan} has been discussed in \cite{Guerin:02}.

The starting point of our analysis is the geometric interpretation of the covariance matrix
\begin{equation}
\label{eq:a:pi} 
a(v)=\vert v\vert^2 \Pi(v),\quad \Pi(v)=\bigl( \text{Id}_N -\frac{v \otimes v}{|v|^2}\bigr), \quad v \in \R^N\backslash\{ 0\},
\end{equation}
where, for $v\neq 0$, $\Pi(v)$ is the orthogonal projection onto $v^\perp$. Indeed,
the key observation is that $a(v) $ also reads as the covariance matrix of the image of $v$ by an antisymmetric standard Gaussian matrix of dimension $N\times N$:
\begin{enumerate}
\item[(5)] 
 changing now the previous $W=((W^i)_{1\le i\le N}) $ into $W=((W^{i,j})_{1 \le i,j\le N}) $ where the $(W^{i,j})_{1 \leq i,j \leq N} $ are independent Gaussian white noises with covariance measure $dsd\alpha$ on $\R^+\times [0,1]$, 
\end{enumerate} 
it holds:
 \begin{equation*}
\begin{split}
\frac 12\E\bigl[ \bigl((W-W^\top)(ds,d\alpha)v \bigr) \otimes \bigl(
(W-W^\top)(ds,d\alpha)v\bigr) \bigr]=a(v) ds d\alpha,\quad v\in \R^N.
\end{split}
 \end{equation*}
The proof is just a consequence of the fact 
 \begin{equation}
 \label{eq:cove:W moins W bar}
\begin{split}
&\frac 12 \sum_{k,\ell=1}^N \E\bigl[ (W-W^\top)_{i,k}(ds,d\alpha)v_{k} \bigl(
(W-W^\top)_{j,\ell}(ds,d\alpha)v_{\ell}\bigr) \bigr]
 \\
&=  \sum_{k,\ell=1}^N \bigl(  \delta_{(i,k)}^{(j,\ell)} -
\delta_{(i,k)}^{(\ell,j)}  \bigr) v_{k} v_{\ell} ds d\alpha
= \bigl( \delta_{i}^j \vert v \vert^2 - v_{i} v_{j} \bigr) ds d\alpha = \bigl( a(v) \bigr)_{i,j} ds d\alpha,
 \end{split}
 \end{equation}
 where we have used the Kronecker symbol in the second line.

We derive the following result, which is at the core of the proof:

\begin{lemma} \label{l1}
Given the process $(Y_{t})_{t \geq 0}$, solution to Equation 
\eqref{lan}, consider the solution $(X_t)_{t \geq 0}$ to the 
SDE
\begin{equation}  \begin{split} \label{lan3}
&X_t = X_{0}
\\
&+ \int_0^t \int_0^1 \frac{W-W^{\top}}{2^{1/2}}(ds, d\alpha) (X_s-Y_s(\alpha))   
-(N-1)\int_0^t \int_0^1 (X_s-Y_s(\alpha)) d\alpha ds.
\end{split}
\end{equation}
Then, $(X_{t})_{t \geq 0}$ has the same law as $(Y_{t})_{t \geq 0}$ and thus as the solution of
the Landau SDE. 
\end{lemma}

\begin{proof}
The proof follows from a straightforward identification of the 
bracket (in time) of the martingale part with $\bar{a}(t,X_{t})dt$. 
\end{proof}

The representation \eqref{lan3} is linear and therefore factorizes through the resolvent. Namely, 
\begin{lemma}
\label{LEMME_RES}
The solution $(X_t)_{t \geq 0}$ to \eqref{lan3} admits the following representation
\begin{equation} \label{lan4}
X_t=Z_t \left[X_0- \int_0^t \int_0^1 Z_s^{\top} \frac{W-W^{\top}}{2^{1/2}}(ds, d\alpha) Y_s(\alpha) \right],
\end{equation}
where letting
\begin{equation}
\label{DEF_B}
B_{t} = 2^{-1/2}\int_{0}^t \int_{0}^1 [W-W^{\top}](ds,d\alpha),
\quad t \geq 0,
\end{equation}
the process $(Z_{t})_{t \geq 0}$ solves the SDE: 
\begin{equation}  \label{resolvent3}
Z_t=\textnormal{Id}_{N}+\int_0^t dB_s Z_s 
-(N-1)\int_0^t Z_s ds=\textnormal{Id}_{N}+\int_0^t dB_s \circ Z_s,
\end{equation}
where $ dB_s \circ$ denotes the Stratonovitch integral. 
\end{lemma}
The proof follows from a straightforward application of It\^o's formula, noticing that 
the bracket 
$\int_0^1 [W-W^{\top}](dt, d\alpha) Z_t \cdot
\int_0^1 Z_t^{\top} [W-W^{\top}](dt, d\alpha) Y_t(\alpha)$ is equal to 
\begin{equation}
\label{eq:cov:B:B bar}
= \int_0^1 ( [W-W^{\top}]\cdot [W-W^{\top}])(dt, d\alpha) Y_t(\alpha)
\\
=-2(N-1)\biggl( \int_0^1 \text{Id}_{N}\,  Y_t(\alpha) d\alpha \biggr) dt,
\end{equation}
since $[W \cdot W]_{i,j}(dt,d\alpha) = 
\sum_{k=1}^N [W_{i,k} \cdot W_{k,j}](dt,d\alpha)
= \delta_{i}^j dt d\alpha$
and $[W \cdot W^{\top}]_{i,j}(dt,d\alpha) = 
\sum_{k=1}^N [W_{i,k} \cdot W_{j,k}](dt,d\alpha)
= 
N
\delta_{i}^j dt d\alpha$.

The main feature is that $(B_{t})_{t \geq 0}$ is a Gaussian process (with values in $\R^{N \times N}$) with 
$\E [ B_{t}^{i,k} B_{t}^{j,\ell} ]
= t (\delta_{(i,k)}^{(j,\ell)}- \delta_{(i,k)}^{(\ell,j)})$ as covariance. 
In particular, $((B^{i,j}_{t})_{1 \leq i < j \leq N})_{t \geq 0}$ is a standard Brownian motion with values in $\R^{N(N-1)/2}$. 
The matrix valued process $B$ thus corresponds to the Brownian motion on the set ${\mathcal A}_{N}(\R)$ of antisymmetric matrices. 
Recalling that ${\mathcal A}_{N}(\R)$ 
is the Lie algebra of the special orthogonal group, 
this allows to identify $(Z_{t})_{t \geq 0}$ with the right Brownian motion on $\SOR{N}$ (see e.g. Chapter V in Rogers and Williams \cite{roge:will:85} and Chapter VII in Franchi and Le Jan \cite{fran:leja:12}). 

\subsection{Conditional representation of the transition density}
Throughout the paper, we shall assume that the centering condition
\begin{equation}
\label{CENTERING}
\E[X_0]=0
\end{equation}
is in force. Actually, there is no loss of generality since, 
whenever $\E[X_{0}] \not =0$,
\eqref{lan3} ensures that, 
for all $t\ge 0, \E[X_t]=\E[X_0] $ and that $(X_t-\E[X_t])_{t\ge 0}$ also solves the equation.

The main representation of the conditional density is then the following:
\begin{prop}
\label{exp}
Assume that $X_0$ is not a Dirac mass and is centered. Then
for all $t>0$, the conditional law of $X_{t}$ given $X_{0}=x_0$ has a density, which can be expressed as
\begin{equation}
\label{eq:25:12:5}
f_{x_0}(t,v) 
= \E \biggl[ (2 \pi)^{-N/2} {\rm det}^{-1/2} ( C_t)
\exp \left( - \frac{1}{2} \bigl\langle v - Z_{t} x_{0}, C_t^{-1} \bigl(v - Z_{t} x_{0} \bigr) \bigr\rangle \right) 
\biggr], 
\end{equation}
for all $v \in \R^N$,
where 
\begin{equation}
\label{eq:cove}
C_{t} = \int_{0}^t Z_{t}Z_{s}^{\top} \bigl(\E[|X_{s}|^2]{\rm Id}_N-\E[X_s \otimes X_{s}] \bigr) \bigl(Z_{t}
{Z}_{s}^{\top}\bigr)^{\top}ds. 
\end{equation}
\end{prop}
The proof of Proposition \ref{exp} is postponed to Section \ref{SEC_DENS}. 

From the above expression of the (stochastic) covariance matrix $C_t$, we introduce the (deterministic) matrix
\begin{equation}
\label{DEF_LAMBDA}
\Lambda_s:=\E[|X_{s}|^2]{\rm Id}_N-\E[X_s  \otimes X_{s}].
\end{equation}

The matrix $\Lambda_s$ then  plays a key role for the control of the non-degeneracy of the diffusion matrix 
$\bar a(s,v)$, which, by \eqref{eq:a:pi}, reads
\begin{equation*}
\bar{a}(s,v) 
=\int_{\R^N}a(v-v_*)f(s,v_*)dv_*=\E \bigl[|X_s-v|^2{\rm Id}_N-(X_s-v) \otimes (X_s-v) \bigl],\quad v\in \R^N. 
\end{equation*}
Since, for all $s\ge 0$, $\E[X_s]=\E[X_0]=0$, we get that for all $v\in \R^N $:
\begin{equation*}
\bar a(s,v)=\Lambda_s-\bigl(2\E[\langle X_s,v\rangle ] {\rm Id}_N- \E[X_s  \otimes v +v \otimes X_s ]
\bigr)+a (v)=\Lambda_s+a (v),
\end{equation*}
so that 
\begin{equation}
\forall \xi \in \R^N,\ \langle \bar a(s,v)\xi,\xi\rangle=\langle\Lambda_s\xi,\xi\rangle+ \langle a(v)\xi,\xi\rangle \ge \langle\Lambda_s\xi,\xi\rangle,\label{DOM_A_LAMBDA}
\end{equation}
where we used that $a $ is positive semidefinite for the last inequality.

The behavior of $\Lambda_s $ can be summarized with the following result.

\begin{prop}
\label{prop:24:7:1}
Assume that $X_{0}$ is not a Dirac mass and is centered. Then,
for any $t>0$, and 
for all $\xi \in \R^N$,
\begin{equation*}
\begin{split}
\Psi_N(t,\overline{\lambda})
\vert \xi \vert^2
\E \left[ \vert X_{0} \vert^2 \right]
\le \langle \xi, \Lambda_t
 \xi \rangle 
 \leq \Psi_N(t,\underline{\lambda}) 
 \vert \xi \vert^2
\E \left[ \vert X_{0}  \vert^2 \right].
\end{split}
\end{equation*}
where for all $(t,\beta)\in \R^+\times [0,1] $,
$$\Psi_N(t,\beta):= (1-1/N)(1-\exp(-2Nt))+(1-\beta)  \exp(-2Nt)  ,$$ 
and
\begin{equation*}
0\le \underline{ \lambda}
:=\inf_{\xi\in \R^N,|\xi|=1} \frac{\E[|\langle \xi,X_0
\rangle|^2]}{\E[|X_0
|^2]}\le 
\sup_{\xi\in \R^N,|\xi|=1} \frac{\E[|\langle \xi,X_0
\rangle|^2]}{\E[|X_0
|^2
]}= : \overline{\lambda}
\le 1.
\end{equation*}
\end{prop}

Proposition \ref{prop:24:7:1} will be proved in the next section. For any $t >0$, 
it provides a lower bound for the spectrum of $\Lambda_{t}$. There are two cases.  
If $\overline{\lambda} < 1$, letting $\overline{\eta} := 
(1-\overline{\lambda}) \wedge (1-1/N) > 0$
(with the standard notations $a \wedge b :=\min(a,b)$
and $a \vee b := \max(a,b)$), it holds that, 
for any $t\geq 0$ and $\xi \in \R^N$,
\begin{equation} \label{eli}
 \langle \xi,\Lambda_t \xi \rangle \geq
\overline{\eta} \vert \xi \vert^2\E \left[ \vert X_{0} 
\vert^2 \right].
\end{equation}
so that $\Lambda_{t}$ is \textit{non-degenerate}, uniformly in time and space.

If $\overline{\lambda} =1$, i.e. $X_0 $ is embedded in a line, then for any $t> 0$ and $\xi \in \R^N$,
\begin{equation*}
\langle \xi,\Lambda_t \xi \rangle \geq
(1-1/N) ( 1 - \exp(-2Nt)  ) \vert \xi \vert^2
\E \left[ \vert X_{0} 
\vert^2 \right],
\end{equation*}
so that $\Lambda_{t}$ is non-degenerate in positive time, uniformly on any $[\varepsilon,+\infty) \times \R^N$, 
$\varepsilon >0$. 
For $t$ small, the lower bound for the spectrum behaves as $2(N-1)t$, so that $\Lambda_{t}$
\textit{degenerates} in small time.

In the following, we will call the case $\overline{ \lambda}<1$ (resp. $\overline{\lambda}=1$)
\textit{non degenerate} 
(resp. \textit{degenerate}).
\begin{rmk}
Equations \eqref{DOM_A_LAMBDA} and \eqref{eli} entail and extend to arbitrary dimension the previous non-degeneracy result of Desvillettes and Villani \cite{DesVillani:00} (Proposition 4) on the diffusion matrix  $\bar a$.
\end{rmk}

\subsection{Estimates in the non-degenerate case}

When $\overline{\lambda} <1$, 
the spectrum of $C_{t}$ in 
\eqref{eq:cove} can be easily controlled 
since $Z_{t} Z_{s}^{-1} = Z_{t} Z_{s}^{\top} \in \SOR{N}$. 
In such a case, we then obtain from \eqref{eq:25:12:5} the following first result for the conditional density of the Landau SDE:
\begin{thm} \label{bds}
Assume that $X_0$ is not a Dirac mass, is centered with  variance 1, 
and its law is not supported on a line.
Then, for all $t>0$ and $v \in \R^N$,
\begin{equation*}
(2 \pi \underline{\eta} t )^{-N/2}
\E\left [\exp \left( - \frac{\vert v-Z_tx_0 \vert^2}{2\overline{\eta} t }  \right) \right]\leq f_{x_0}(t,v) \leq (2 \pi 
\overline{\eta} t)^{-N/2}
\E\left[\exp \left( - \frac{\vert v -Z_tx_0\vert^2}{2\underline{\eta} t}  \right)\right],
\end{equation*}
where 
$\overline{\eta} := (1-\overline{\lambda}) \wedge (1-1/N) \leq \underline{\eta} := 
(1-\underline{\lambda}) \vee (1-1/N)$.
\end{thm}

\begin{rmk}
\label{DENS_EST_RK_ISO}
Observe that, since $(Z_s)_{s\ge 0}$ defines an isometry, 
the off-diagonal cost $\vert v-Z_tx_0 \vert^2$ may be rewritten 
$\vert Z_{t}^{\top}v- x_0 \vert^2$.

This formulation may be more adapted than the previous one 
when integrating the conditional density with respect to the initial law of $X_0$. 
\end{rmk}
Now, exploiting the \textit{Aronson like} heat kernel bounds for the marginal 
density of 
the rotation process $(Z_t)_{t \geq 0}$, see e.g. Varopoulos \textit{et al.} \cite{varo:salo:coul:92} or Stroock \cite{stro:06}, we actually derive in Section 
\ref{SEC_NON_DEG}
the following control:

\begin{thm}[Explicit bounds for the conditional density]
\label{bds_EXPL_UE}
Under the assumptions of Theorem \ref{bds},  there exists $C:=C(N)\ge 1 $
 such that, for all $t>0$, $x_{0},v\in \R^N$, 
\begin{equation}
\label{LOSS_MASS_UE}
\frac{\delta_t^{N-1}}{C t^{N/2}} \exp \bigl( - C \frac{I}{t} \bigr) \leq f_{x_{0}}(t,v)
\leq \frac{C \delta_t^{N-1}}{t^{N/2}} \exp \bigl( - \frac{I}{Ct} \bigr),
\end{equation}
where
\begin{equation*}
\delta_t = 
\frac{1 \wedge  \Bigl(  \frac{\displaystyle t^{1/2}}{\displaystyle 
1 \vee (
\vert x_{0} \vert \wedge \vert v \vert)} 
\Bigr)}{1 \wedge t^{1/2}}, \quad \textrm{and} \quad
I = \big\vert \vert v \vert- \vert x_0 \vert \big\vert^2+
(1 \wedge \vert v \vert \wedge \vert x_0 \vert)^2
\vert \frac{v}{|v|}-\frac{x_0}{\vert x_0\vert}\vert^2. 
\end{equation*}
If $|x_0|\wedge |v|\le 1$, then
$\delta_t$ is equal to $1$ and $I$ can be chosen as $I=|x_0-v|^2$,
which corresponds to a \textit{usual} Gaussian estimate. 
\end{thm}

We stress the fact that the above bounds are sharp.
The contribution in $\vert|v|-|x_0|\vert^2$ in $I$ corresponds to a `radial cost'
and the contribution in 
$\vert v/|v|- x_0/\vert x_0\vert\vert^2$ to a `tangential cost'. 
The term $(1 \wedge \vert v \vert \wedge \vert x_{0} \vert)^2$ reads as the inverse of the 
variance along tangential directions. It must be compared with the variance along tangential directions
in a standard Gaussian kernel, the inverse of which is of order
$(\vert v \vert \wedge \vert x_{0}\vert)^2$ as shown in Remark \ref{proj2} below.
This says that, when $\vert x_{0} \vert$ and $\vert v \vert$ are greater than $1$, 
$f_{x_{0}}(t,v)$ is super-diffusive in the tangential directions. This is in agreement with the 
observations made in Introduction: The non-Gaussian regime of the density for $x_{0}$ large occurs 
because of the super-diffusivity along iso-radial curves. 

Anyhow, it is worth mentioning that the two-sided bounds become Gaussian
when $t$ tends to $\infty$. Indeed, noting that $\delta_{t} \rightarrow 1$ as 
$t \rightarrow \infty$ and that the tangential cost 
$(1 \wedge \vert v \vert \wedge \vert x_0 \vert)^2
\vert v/|v|- x_0/\vert x_0\vert\vert^2$
is bounded by 4, \eqref{LOSS_MASS_UE} yields
\begin{equation}
\label{eq:asymp:two:sided}
\frac{1}{C t^{N/2}} \exp \bigl( - C \frac{\bigl\vert \vert x_{0}\vert - \vert v \vert \bigr\vert^2}{t}
 \bigr) \leq f_{x_{0}}(t,v)
\leq \frac{C}{t^{N/2}} \exp \bigl( - \frac{\bigl\vert \vert x_{0}\vert - \vert v \vert \bigr\vert^2}{Ct} \bigr),
\end{equation} 
for $t$ large enough (with respect to $\vert x_{0}\vert$, uniformly in $\vert v\vert$)
and for a new constant $C$ (independent of $\vert x_{0}\vert$ and $\vert v\vert$). This coincides with the asymptotic behavior of the 
$N$-dimensional Gaussian kernel: In the Gaussian regime, the 
variance along the tangential directions is $(\vert v\vert \wedge \vert x_{0} \vert)^2$, which is less than 
$\vert x_{0} \vert^2$ and which shows, in the same way as in \eqref{eq:asymp:two:sided}, that the 
Gaussian tangential cost is also small in front of $t$, uniformly in 
$v$. However, some differences persist asymptotically when $\vert x_{0}\vert$ is large.
Due to the super-diffusivity of the tangential directions in the Landau equation, 
the Landau tangential cost 
decays faster than the Gaussian one. Intuitively, the reason 
is that the `angle' of the Landau process $(X_{t})_{t \geq 0}$ reaches the uniform distribution on the sphere 
at a quicker rate than in the Gaussian regime. Clearly, 
the fact that the system forgets the initial angle of $x_{0}$ in long time could be recovered
from Theorem \ref{bds} by replacing (at least formally) 
$Z_{t}$ by a uniformly distributed random matrix on $\SOR{N}$.

Of course, when the initial mass is already uniformly 
distributed along the spheres centered at $0$, 
the marginal density of $(X_{t})_{t \geq 0}$ already behaves in finite time as if the transition 
density was Gaussian. We illustrate this property in the following corollary (the proof of which is deferred to the next section):

\begin{cor}
\label{thm:24:12:1}
Assume that $X_{0}$ admits an initial density of the radial form:
\begin{equation*}
f_{0}(x_{0}) = f(\vert x_{0} \vert),
\end{equation*}
for some Borel function $f : \R_{+} \mapsto \R_{+}$. Then, 
we can find a constant $C:=C(N) \geq 1$ such that, for all $t>0$, 
\begin{equation}
\label{eq:comparaison:noyaux:d}
\begin{split}
\frac{1}{C t^{N/2}}
\int_{\R^N} f_{0}(x_{0}) g_{N} \bigl( C\frac{x_{0} - v}{t^{1/2}} \bigr) dx_0
\leq
f_{t}(v) \leq \frac{C}{t^{N/2}}
\int_{\R^N} f_{0}(x_{0}) g_{N} \bigl( \frac{x_{0} - v}{Ct^{1/2}} \bigr) dx_0, 
\end{split}
\end{equation}  
where $g_{N}$ denotes the standard Gaussian kernel of dimension $N$ and where $f_{t}$
 is the solution of the Landau equation, which here reads
\begin{equation*}
f_{t}(v) = \int_{\R^N} f_{0}(x_{0}) f_{x_0}(t,v) dx_{0}. 
\end{equation*}
\end{cor}

\vspace{2pt}

To conclude this subsection, notice that the Gaussian regime (that corresponds to
$\vert x_{0} \vert \wedge \vert v \vert \leq 1$ in the statement of 
Theorem \ref{bds_EXPL_UE})  can be derived from \eqref{LOSS_MASS_UE} using the following Lemma and Remark.
\begin{lemma} \label{proj}
Let $x_0\in \R^N$ be given and
$\Pi_{B_N(0,|x_0|)}$ denote the orthogonal projection from $\R^N$ onto the ball $B_{N}(0,\vert x_{0}\vert )$ of center $0$ and radius $\vert x_{0} \vert$. 
Then, for all $v \in \R^N$ such that $\vert x_0 \vert <\vert v \vert$,
\begin{equation*}
\begin{split}
 &\vert v - \Pi_{B_N(0,|x_0|)}(v) \vert^2+\vert \Pi_{B_N(0,|x_0|)}(v) - x_{0} \vert^2
 \\
 &\hspace{15pt} \leq \vert v - x_{0} \vert^2 \leq  2\vert v - \Pi_{B_N(0,|x_0|)}(v) \vert^2+2 \vert \Pi_{B_N(0,|x_0|)}(v) - x_{0} \vert^2. 
\end{split}
\end{equation*}
\end{lemma}

\begin{proof}
We write
\begin{equation*}
\begin{split}
\vert v -  x_{0} \vert^2 &= \vert v - \Pi_{B_N(0,|x_0|)}(v) + \Pi_{B_N(0,|x_0|)}(v) -  x_{0} \vert^2
\\
&= \vert v - \Pi_{B_N(0,|x_0|)}(v) \vert^2+ \vert \Pi_{B_N(0,|x_0|)}(v) - x_{0} \vert^2 \\
&\hspace*{10pt}+ 
2 \langle v - \Pi_{B_N(0,|x_0|)}(v), \Pi_{B_N(0,|x_0|)}(v) - x_{0} \rangle.  
\end{split} 
\end{equation*}
Now $\langle v - \Pi_{B_N(0,|x_0|)}(v), \Pi_{B_N(0,|x_0|)}(v) - x_{0} \rangle \geq 0$, by orthogonal projection on a closed convex subset, and the lower bound follows. By convexity, we obtain the upper bound.
\end{proof}

\begin{rmk} \label{proj2} Let us consider two given points $x_0,v\in \R^N $ such that $|x_0|\le |v| $.
Noticing that $\Pi_{B_N(0,|x_0|)}(v) = ( \vert x_{0} \vert / \vert v \vert) v$
and then that $\vert v - \Pi_{B_N(0,|x_0|)}(v) \vert = \vert v \vert - \vert x_{0} \vert$, we deduce from Lemma \ref{proj} that 
\begin{equation*}
\big\vert  \vert v \vert - \vert x_{0} \vert \big\vert^2+ 
\vert x_0\vert^2 \bigg\vert \frac{v}{\vert v \vert}
- \frac{x_{0}}{\vert x_0 \vert} 	\bigg\vert^2\leq \vert v - x_{0} \vert^2 \leq 2 \big\vert  \vert v \vert - \vert x_{0} \vert \big\vert^2 + 
2 \vert x_0\vert^2\bigg\vert \frac{v}{\vert v \vert}
 - \frac{x_{0}}{\vert x_0 \vert} 	\bigg\vert^2.   
\end{equation*}
In particular, when $|x_0|\le 1 $ we derive from \eqref{LOSS_MASS_UE} in Theorem \ref{bds_EXPL_UE} the usual two-sided Gaussian estimates. Now, if  $|v|\le |x_0| $ and $|v|\le 1 $, this still holds by symmetry. 
\end{rmk}

\subsection{Estimates in the degenerate case}
We now discuss the case when the initial condition lies in a straight line, which by rotation invariance can be assumed to 
be the first vector $e_{1}$ of the canonical basis. By Proposition \ref{prop:24:7:1}, we 
already know that the matrix $\Lambda_{t}$ (see \eqref{DEF_LAMBDA})
 driving the ellipticity of the covariance matrix 
$C_{t}$ (see \eqref{eq:cove}) becomes non-degenerate in positive time. 
This says that, after a positive time $t_{0}$, the system enters the same regime as the one discussed in Theorem 
\ref{bds_EXPL_UE}, so that the transition density of the process satisfies, after $t_{0}$, the bounds 
\eqref{LOSS_MASS_UE}. Anyhow, this leaves open the small time behavior of the transition kernel of the process.

Here, we thus go thoroughly into the analysis and specify both the on-diagonal rate of explosion and the off-diagonal decay of the conditional density in small time. Surprisingly, we show that the tail of the density looks much more like an exponential distribution rather than a Gaussian one. Precisely, we show that the off-diagonal decay of the density is of Gaussian type for `untypical' values only, which is to say that, for values where the mass is effectively located, the decay is of exponential type. Put it differently, the two-sided bounds we provide for the conditional density read as a mixture of exponential and Gaussian distributions.

\begin{thm}
\label{thm:bd:degenerate}
Assume that the initial distribution of $X_{0}$ is compactly supported by $e_{1}$, i.e. there exists $C_0>0$ such that $X_0\in [-C_0e_1,C_0e_1]$  a.s.
Then, 
there exists $
C:=C(C_0) >1$ such that, for $t \in (0,1/C]$: 
\begin{equation*}
\frac{1}{C t^{(N+1)/2}} \exp \bigl( - C I(t,x_{0},v) \bigr) \leq f_{x_{0}}(t,v)
\leq \frac{C}{t^{(N+1)/2}} \exp \bigl( - \frac{I(t,x_{0},v)}{C} \bigr),
\end{equation*}
where  
%
%
%
%
%
%
%
%
$$\displaystyle I(t,x_{0},v) =
\frac{\vert v^1 - x_{0}^1 \vert}{t} + 
\frac{ \vert v^1 - x_{0}^1 \vert^2}{t} 
+ \sum_{i=2}^{N}\frac{|v^i|^2}t.$$

%

\end{thm}

The reason why the conditional density follows a mixture of exponential and Gaussian rates may be explained as follows in the simplest case when $x_{0}=0$.
The starting point is formula \eqref{eq:25:12:5} in Proposition \ref{exp}. When the initial condition is degenerate, the conditional covariance matrix $C_{t}$ in \eqref{eq:cove} has two scales.
As shown right below, the eigenvalues of $C_{t}$ 
along the directions $e_{2},\dots,e_{N}$ are of order $t$ whereas 
the eigenvalue $\lambda_{t}^1$ of $C_{t}$ along the direction $e_{1}$ is of order $t^2$
with large probability. Anyhow, with exponentially small probability, 
$\lambda_{t}^1$ is of order $t$: Precisely, the probability that it is of order $\xi t$ has logarithm of order $- \xi/t$
when $\xi \in (0,1)$. Such large deviations of $\lambda_{t}^1$ follow from large deviations of 
$(Z_{s})_{0 \leq s \leq t}$ far away from the identity.
This rough description permits to compare the contributions of typical and rare events
in the formula \eqref{eq:cove} for the density  $f_{x_{0}}(t,v)$, when 
computed at a vector $v$ parallel to the direction $e_{1}$. 
On typical scenarios,
the off-diagonal cost $\langle C_{t}^{-1}v,v\rangle$ in the exponential appearing in \eqref{eq:cove} is of order $\vert v \vert^2/t^2$. 
In comparison with, by choosing $\xi$ of order $\vert v \vert$, the events associated 
with large deviations of $C_t$ generate an off-diagonal cost $\langle C_{t}^{-1}v,v\rangle$ of order $\vert v\vert/t$ with 
an exponentially small probability of logarithmic order -$\vert v\vert/t$: 
The resulting contribution in the off-diagonal decay 
is order $\vert v \vert/t$, 
which is
 clearly smaller than $\vert v \vert^2/t^2$. 
This explains the exponential regime of $f_{x_{0}}(t,v)$. The Gaussian one follows from a threshold phenomenon: as $(Z_{s})_{0 \leq s \leq t}$
takes values in $\textrm{SO}_{N}(\R)$, there is no chance for its elements to exceed $1$ in 
norm. Basically, it means that, when $\vert v \vert$ is large, the best choice for $\xi$ is not $\vert v\vert$ but $1$: The corresponding off-diagonal cost is $\vert v\vert^2/t$, which occurs with probability of logarithmic order $-1/t$. This explains the Gaussian part of 
$f_{x_{0}}(t,v)$.

In the case when the conditioned initial position $x_{0}$ is not zero, specifically when it is far away from $0$, things become much more intricate as the transport of the initial position $x_{0}$ by $Z_{t}$ affects the density. This is the reason why we consider a compactly supported initial condition.  To compare with,
notice that, in the non-degenerate case, \eqref{LOSS_MASS_UE} gives Gaussian estimates 
when $x_{0}$ is restricted to a compact set. 
This is exactly what the statement of 
Theorem \ref{bds_EXPL_UE} says when $\vert x_{0}\vert \leq 1$, the argument working in the same way when 
$\vert x_{0}\vert \leq C_{0}$, for some $C_{0} > 1$.


\section{Conditional density of the Landau SDE: Derivation and Properties}
\label{SEC_DENS}

\subsection{Proof of Proposition \ref{exp}}
\label{SEC_RES_REP}
We claim
\begin{lemma} \label{l4}
\label{lem:independence}
Recall $X_0$ is centered. Letting
\begin{equation} \label{barb1}
\bar{B}_{t} = 2^{-1/2}\int_{0}^t \int_{0}^1 [W-W^{\top}](ds,d\alpha) Y_{s}(\alpha),
\end{equation}
the processes $(B_t)_{t\ge 0} $ and $(\bar B_t)_{t\ge 0} $  are independent.
Also, the processes $(Z_{t})_{t \geq 0}$ and $(\bar{B}_{t})_{t \geq 0}$ are independent. 
\end{lemma}
  
\begin{proof}
 We know that setting $\tilde Z_t:=\exp((N-1)t)Z_t $ then
\begin{equation*}
\begin{split}
\tilde Z_{t} &=\text{Id}_{N} + B_{t} + \int_{0}^t dB_{s} B_{s} + \dots
=\text{Id}_{N} + \sum_{n \geq 1}
\int_{0 \leq t_{n} \leq \dots \leq t_{1} \leq t} dB_{t_{1}} dB_{t_{2}} \dots dB_{t_{n}}.
\end{split}
\end{equation*}
Hence it suffices to show that $B$ and $\bar{B}$ are independent. 
As both are Gaussian processes, this can be easily proved by computing their covariance which turns out to be zero if 
$X_0$ is centered, see \eqref{eq:cov:B:B bar}.
\end{proof}

Recalling that we can rewrite $X_t$ as
\begin{equation} \label{lan7}
X_t=Z_t \biggl[X_0 -\int_0^t Z_s^{\top} d\bar{B}_s\biggr], \quad t \geq 0,
\end{equation}
$X_{0}$ being independent of $(B_{t},\bar{B})_{t \geq 0}$,
and using \eqref{eq:cove:W moins W bar} to compute the covariance matrix 
of the Gaussian process $(\bar{B}_{t})_{t \geq 0}$:
\begin{equation} \label{barb2}
\frac{d}{dt}
\E \bigl[ \bar{B}_{t} \bar{B}_{t}^{\top}
\bigr] = \int_{0}^1 \bigl\{\textrm{Id}_{N} \vert Y_{t}(\alpha) \vert^2 - 
Y_{t}(\alpha)  \otimes Y_{t}(\alpha) \bigr\}d\alpha = 
\E[\vert X_t \vert^2] \text{Id}_N  - \E[X_t \otimes X_t] =\Lambda_t,
\end{equation}
the existence of the transition density and the representation \eqref{eq:25:12:5} are 
 direct consequences of (\ref{lan7}) and Lemma \ref{l4}.
This proves Lemma \ref{LEMME_RES}.

\subsection{Additional properties on the resolvent process}
\label{EXTRA_PROP_Z}
We give in this paragraph some additional properties on the process $Z$ that are needed for the derivation of the density estimates.
We will make use of the following lemma whose proof can be found in Franchi and Le Jan \cite{fran:leja:12}, see Theorem VII.2.1 and Remark VII.2.6.
\begin{lemma}
\label{lem:27:12:1}
Given $t>0$, the process $(Z_{t} Z_{t-s}^{\top})_{0 \leq s \leq t}$ has the same law as the process $(Z_{s})_{0 \leq s \leq t}$.
\end{lemma}


\subsection{Proof of Proposition \ref{prop:24:7:1}} 
\label{SEC_COV}
Recall from \eqref{lan3} that the expectation is preserved, i.e.
\begin{equation}
\label{eq:22:8:10}
\E[X_{t}] = \E[X_{0}], \qquad \text{ for all }  t \geq 0.
\end{equation}
Since we also assumed that $\E[X_{0}]=0$, the process $(X_{t})_{t \geq 0}$ is centered. 
The point is then to compute
\begin{equation*}
\langle \Lambda_t\xi,\xi\rangle=\E \left[ \vert \xi \vert^2 \vert X_{t} \vert^2 
- \langle \xi, X_{t}  \rangle^2 \right],
\quad \xi \in \R^N, \ t \geq 0. 
\end{equation*}
Noting that ${\rm Trace} [a(v)] = (N-1) \vert v \vert^2$, for $v \in \R^N$, we get that
\begin{equation*}
\begin{split}
\frac{d}{dt} \E \left[ \vert X_{t} \vert^2 \right] 
&= \int_{0}^1 \E \left[{\rm Trace} \left[ a(X_{t}-Y_{t}(\alpha))\right] \right]d\alpha 
- 2(N-1) \int_{0}^1 \E \left[ \langle X_{t},X_{t}- Y_{t}(\alpha) \rangle \right] d\alpha 
\\
&= (N-1) 
\int_{0}^1 \E \left[ \vert X_{t}-Y_{t}(\alpha) \vert^2 \right] d\alpha 
- 2(N-1)  
\E \left[ \vert X_{t} \vert^2 \right]
= 0.
\end{split}
\end{equation*}
Therefore, the energy is preserved:
\begin{equation}
\label{eq:24:7:1} 
\E \left[ \vert X_{t} \vert^2 \right] = \E \left[ \vert X_{0} \vert^2 \right],
\qquad \text{ for all }  t \geq 0.
\end{equation}
Moreover, using the expression (\ref{lan}) of the Landau SDE (which implies that, 
just in the equation right below, $W$ becomes again an $N$-dimensional space-time white noise), we see that
\begin{equation*}
d \langle \xi,X_{t} \rangle = \int_{0}^1 \langle \sigma^{\top}(X_t-Y_t(\alpha))\xi, W(dt,d\alpha) \rangle - (N-1) 
 \int_{0}^1 \langle \xi, \left( X_{t} - Y_{t}(\alpha) \right) \rangle d \alpha dt.
\end{equation*}
Since $\vert \sigma^{\top}(y) \xi \vert^2 
= \vert \xi \vert^2 \vert y \vert^2 - \langle \xi,y \rangle^2$, we have
\begin{equation*}
\begin{split}
&\frac{d}{dt} \E \left[ \langle \xi,X_{t} \rangle^2 \right] \\
&= \int_{0}^1 \E \left[ \vert 
\sigma^{\top}\left(X_{t} - Y_{t}(\alpha)\right)\xi \vert^2 \right] d \alpha  
-  2(N-1)
\E  \int_{0}^1 \langle \xi, X_{t} \rangle 
\langle \xi, \left( X_{t} - Y_{t}(\alpha) \right) \rangle d \alpha 
\\
&= \vert \xi \vert^2 
 \int_{0}^1 \E \left[ \vert X_{t} - Y_{t}(\alpha) \vert^2 \right] d\alpha  
- \int_{0}^1 \E \left[ \langle \xi, X_{t} - Y_{t}(\alpha)  \rangle^2 \right] d\alpha  
-2(N-1)\E  \left[ \langle \xi, X_{t} \rangle^2 \right] 
\\
&= 2 \vert \xi \vert^2
\E \left[ \vert X_{t} \vert^2 \right] 
-  2N
\E  \left[ \langle \xi, X_{t} \rangle^2 \right].
\end{split}
\end{equation*}
From \eqref{eq:24:7:1}, we deduce that
\begin{equation*}
\frac{d}{dt} \E \left[ \langle \xi,X_{t} \rangle^2 \right]
= 2 \vert \xi \vert^2
\E \left[ \vert X_{0} \vert^2 \right]  
-  2N
\E  \left[ \langle \xi, X_{t} \rangle^2 \right],
\end{equation*}
so that, for any $t \geq 0$, 
\begin{equation*}
\E \left[  \langle \xi,X_{t} \rangle^2 \right]
= \exp(-2Nt) \biggl\{ \E \left[  \langle \xi,X_{0} \rangle^2 \right]
+  2 \int_{0}^t \exp(2Ns)  \vert \xi \vert^2
\E \left[ \vert X_{0} \vert^2 \right]  ds \biggr\}. 
\end{equation*}
Finally,
\begin{equation*}
\E \left[  \langle \xi,X_{t} \rangle^2 \right]
= \exp(-2Nt) \E \left[  \langle \xi,X_{0} \rangle^2 \right]
+ \frac{1}{N} \left[ 1 - \exp(-2Nt) \right] 
 \vert \xi \vert^2 
\E \left[ \vert X_{0} \vert^2 \right]. 
\end{equation*}
Therefore, 
\begin{equation}
\begin{split}
\langle \Lambda_t\xi,\xi\rangle&=\vert \xi \vert^2 \E \left[  \vert X_{t} \vert^2 \right]Ä
- \E \left[  \langle \xi,X_{t} \rangle^2 \right]\\
&=  \frac{1}{N} \left[ N- 1 + \exp(-2Nt) \right] \vert \xi \vert^2
\E \left[ \vert X_{0} \vert^2 \right]
- \exp(-2Nt) \E \left[  \langle \xi,X_{0} \rangle^2 \right].
\end{split}
\label{DEV_COV}
\end{equation}
Plugging the values of $\underline{\lambda}$ and $\overline{\lambda}$
in \eqref{DEV_COV}, we get the announced result.

\subsection{Proof of Corollary \ref{thm:24:12:1}}
By Theorem \ref{bds_EXPL_UE}, the result is straightforward when $\vert v \vert \leq 1$ (as the transition density has a Gaussian 
shape). When $\vert v \vert \geq 1$, the problem can be reformulated as follows. 
Given a constant $C>0$, the point is to estimate
\begin{equation}
\label{eq:qt:1}
\begin{split}
q_{t}(v) := \int_{\R^{N}} &\frac{[ \delta_{t}(\vert x_{0} \vert)]^{N-1}}{ t^{N/2}}
\\
&\hspace{5pt} \times f_{0}(x_{0})
\exp\left(-\frac{C}{t} \left\{\big\vert \vert v \vert- \vert x_0 \vert \big\vert^2
+(1 \wedge \vert x_0 \vert)^2\big\vert \frac{v}{|v|}-\frac{x_0}{\vert x_0\vert} \big\vert^2\right\}\right)
dx_{0},
\end{split}
\end{equation}
where we have let 
$\displaystyle \delta_{t}(\vert x_{0} \vert)  :=  \frac{1\wedge 
 \Bigl( \frac{\displaystyle t^{1/2}}{\displaystyle 1\vee (|x_0|\wedge |v|)} \Bigr)}{1\wedge t^{1/2}}$. 
\vspace{5pt}

By a polar change of variable, we get
\begin{equation}
\label{eq:qt:2}
\begin{split}
q_{t}(v) =\frac1{  t^{N/2}} \int_{0}^{+\infty} &d\rho \, \rho^{N-1}
\bigl[ \delta_{t}(\rho) \bigr]^{N-1}  f_{0}(\rho)\exp\left(-\frac{C}{t}\big\vert \rho - \vert v \vert \big\vert^2\right)
\\
&\hspace{15pt} \times \int_{\Sph^{N-1}}
\exp\left( -\frac{C}t(1 \wedge \rho)^2\big\vert s -\frac{v}{\vert v\vert} \big\vert^2\right)
 d\nu_{\Sph^{N-1}}(s),
\end{split}
\end{equation}
where $\nu_{\Sph^{N-1}}$ denotes the Lebesgue measure on the sphere $\Sph^{N-1}$ of dimension $N-1$.

As we shall make use of its renormalized version below, we normalize $\nu_{\Sph^{N-1}}$, so that 
$\nu_{\Sph^{N-1}}$ reads as a probability measure. Up to a multiplicative constant, the above expression remains
unchanged. In particular, as we are just interested in lower and upper bounds of $q_{t}(v)$, we can 
keep the above as a definition for $q_{t}(v)$, with $\nu_{\Sph^{N-1}}$ being normalized.

Let us now recall the following two-sided heat kernel estimate on ${\mathbb S}^{N-1} $, see e.g. \cite{stro:06}. There exists $C':=C'(N)\ge 1$ such that, for all $t>0$,
\begin{equation}
\label{eq:hkes:sphere}
(C')^{-1}
\leq \frac{1}{\Bigl(1\wedge \displaystyle \frac{t^{1/2}}{1\wedge \rho} \Bigr)^{N-1}}
\int_{\Sph^{N-1}} 
\exp\left(-\frac{C (1 \wedge \rho)^2}{t} \big\vert s -\frac{v}{\vert v\vert} \big\vert^2\right)
d\nu_{\Sph^{N-1}}(s) 
\leq 
C'.
\end{equation}
Therefore, what really counts in the expression of $q_{t}(v)$ is the product
\begin{equation}
\label{eq:qt:4}
 \Bigl( 1\wedge \displaystyle \frac{t^{1/2}}{1\wedge \rho} \Bigr) \delta_{t}(\rho)
= \left\{
\begin{array}{ll}
1\wedge \displaystyle \frac{t^{1/2}}{\rho} \quad
&\text{if} \ \rho \leq 1,
\vspace{4pt}
\\
1\wedge \displaystyle \frac{t^{1/2}}{\rho} \quad
&\text{if} \ 1 \leq \rho \leq \vert v \vert,
\vspace{4pt}
\\
1\wedge \displaystyle \frac{t^{1/2}}{\vert v \vert} &\text{if} \ \rho > \vert v\vert,
\end{array}
\right.
\end{equation}
Up to a redefinition of the function $q_{t}$, it is thus sufficient to consider 
\begin{equation}
\label{eq:qt:3}
\begin{split}
q_{t}(v) &:= \frac{1}{t^{N/2}} \int_{0}^{|v|}    
 f_{0}(\rho)
\exp\left(-\frac{C}{t} \big\vert \rho - \vert v \vert \big\vert^2
\right)\left\{ 1\wedge \frac{t^{1/2}}{\rho}\right\}^{N-1}
\rho^{N-1}d\rho
\\
&\hspace{15pt} +\frac{1}{t^{N/2}} \int_{|v|}^{+\infty}  
 f_{0}(\rho)
\exp\left(-\frac{C}{t} \big\vert \rho - \vert v \vert \big\vert^2
\right)\left\{ 1\wedge \frac{t^{1/2}}{|v|}\right\}^{N-1}
\rho^{N-1}
d\rho.
\end{split}
\end{equation}

Compare now with what happens when the convolution in \eqref{eq:qt:1} is made with respect to the Gaussian kernel. 
Basically $\delta_{t}(\vert x_{0}\vert)$ is replaced by $1$ and 
$1 \wedge \vert x_{0} \vert$ is replaced by $\vert v \vert \wedge \vert x_{0} \vert$ (see 
Remark \ref{proj2}). Equivalently,
$\delta_{t}(\rho)$ is replaced by $1$ and 
$1 \wedge \rho$ by $\vert v \vert \wedge \rho$ in 
\eqref{eq:qt:2}. 
This says that, in \eqref{eq:hkes:sphere}, $1 \wedge \rho$ is replaced 
by $\vert v \vert \wedge \rho$. Then, in \eqref{eq:qt:4}, 
$\delta_{t}(\rho)$ is replaced by $1$ and 
$1 \wedge \rho$ by $\vert v \vert \wedge \rho$, which leads exactly to the same three equalities. This shows that, in the Gaussian regime, the right quantity to consider is also \eqref{eq:qt:3}. 
\section{Proof of the Density Estimates in the Non-Degenerate case}
\label{SEC_NON_DEG}
\subsection{Preliminary results for the  Haar measure and the heat kernel on $\SOR{N}$}
Starting from the representation Theorem \ref{bds}, we want to exploit the \textit{Aronson like} heat kernel estimates 
for the special orthogonal group. Precisely, from 
VIII.2.9 in Varopoulos \textit{et al} \cite{varo:salo:coul:92}, we derive that, for $t>0$, 
the law of  $Z_t$ has a density, denoted by $p_{\SO{N}}(t,\textrm{Id}_{N},\cdot)$, with respect to the probability Haar measure 
$\mu_{\SO{N}}$ of $\SOR{N}$.  
Moreover, there exists a constant $\beta>1$ such that, 
for any $g \in \SOR{N}$ and for all $t>0$:
\begin{equation}
\label{STROOCK_SALOFF}
\begin{split}
&\frac{1}{\beta (1\wedge t)^{N(N-1)/4}}\exp\biggl(-\beta \frac{d_{\SO{N}}^2(\textrm{Id}_{N},g)}{t}\biggr)
\\
&\hspace{15pt} \le p_{\SO{N}}(t,\textrm{Id}_{N},g)\le \frac{\beta}{(1\wedge  t)^{N(N-1)/4}}\exp\biggl(-\frac{d^2_{\SO{N}}(\textrm{Id}_{N},g)}{\beta t}\biggr),
\end{split}
\end{equation}
where $d_{\SO{N}}(\textrm{Id}_{N},g)$ denotes the Carnot distance between $\textrm{Id}_{N}$ and $g$:
\begin{equation*}
d_{\SO{N}}(\textrm{Id}_{N},g) = \inf_{H \in {\mathcal A}_{N}(\R) : e^H=g} \| H \|,
  \end{equation*} 
$\| \cdot \|$ standing for the usual matricial norm on ${\mathcal M}_{N}(\R)$.
Proof of the diagonal rate in \eqref{STROOCK_SALOFF} relies on the 
following volume estimate from Theorem V.4.1 in \cite{varo:salo:coul:92}: 
By compactness of  $\SO{N}(\R)$,
there exists $C_{N} \geq 1$ such that, for all $t>0$,
\begin{equation}
\label{CTR_VOL}
C_{N}^{-1} (1 \wedge t)^{N(N-1)/4} \leq 
\mu_{\SO{N}}\bigl(B_{\SO{N}}(t^{1/2}) \bigr) \leq C_{N} (1 \wedge t)^{N(N-1)/4}, 
\end{equation}
where $B_{\SO{N}}(\rho):=\{ g\in \SOR{N}: d_{\SO{N}}(\textrm{Id}_{N},g)\le \rho \}$, for $\rho >0$, denotes 
the ball of radius $\rho$ and center $\textrm{Id}_{N}$. \color{black}

By local inversion of the exponential, it is well-checked that the Carnot distance is continuous with respect to the standard matricial 
norm on ${\mathcal M}_{N}(\R)$. In particular, by compactness of ${\mathrm{SO}}_{N}(\R)$, it is bounded on the whole group. Actually, we claim:
\begin{lemma}[Equivalence between Carnot distance and matrix norm on the group]
\label{lem:carnot}
There exists a constant $C:=C(N) >1$ such that, for any $g \in {\mathrm{SO}}_{N}(\R)$, 
\begin{equation*}
C^{-1} \| \textrm{\rm Id}_{N} - g \| \leq d_{\SO{N}}(\textrm{\rm Id}_{N},g) \leq C \| \textrm{\rm Id}_{N} - g \|.
\end{equation*}
\end{lemma}
\begin{proof}
We first prove the upper bound. Considering a given $g \in {\mathrm{SO}}_{N}(\R)$, we can assume without any loss of generality that $\| \textrm{Id}_{N}- g \| \leq \varepsilon$, for some arbitrarily prescribed $\varepsilon >0$. Indeed, if $\|\textrm{Id}_{N}-g \| > \varepsilon$, the upper bound directly follows from the boundedness of the Carnot distance on the group. 

Choosing $\varepsilon$ small enough, we can assume that the logarithm mapping on ${\mathcal M}_{N}(\R)$ realizes a diffeomorphism from the 
ball of center $\textrm{Id}_{N}$ and radius $\varepsilon >0$ into some open subset around the null matrix. Then, letting 
$H := \ln(g)$, we deduce from the variational definition of the distance that 
$d_{\SO{N}}(\textrm{Id}_{N},g) \leq  \| H \|$. Writing $H  = \ln(\textrm{Id}_{N} + g-\textrm{Id}_{N})$, we obtain that $\| H \| \leq C \| g - \textrm{Id}_{N} \|$ for some $C:=C(N)$, which proves that  
$d_{\SO{N}}(\textrm{Id}_{N},g) \leq  C \| g - \textrm{Id}_{N} \|$. 

The converse is proved in a similar way. Without any loss of generality, we can assume that $d_{\SO{N}}(\textrm{Id}_{N},g) \leq \varepsilon$, for some 
given $\varepsilon>0$. By the variational definition of the distance, this says that there exists a matrix $H \in {\mathcal A}_{N}(\R)$ such that $\exp(H) = g$ and $d_{\SO{N}}(\textrm{Id}_{N},g) \geq \|H \|/2$, with $\|H\| 
\leq 2\varepsilon$. By the 
Lipschitz property of the exponential around $0$, 
$\| g - \textrm{Id}_{N} \| \leq C \| H \|$ (for a possibly new value of the constant $C$), which yields  
$\| g - \textrm{Id}_{N} \| \leq 2 Cd_{\SO{N}}(\textrm{Id}_{N},g)$. 
\end{proof}

Part of our analysis relies on a specific parametrization of $\SOR{N}$ by elements of 
$\Sph^{N-1} \times \SOR{N-1}$, where $\Sph^{N-1}$ is the sphere of dimension $N-1$. Namely, 
for an element $h \in \SOR{N-1}$, we denote by $L_h$ the element of $\SOR{N}$:
\begin{equation*}
 L_{h}:=\left( \begin{array}{c|ccc} 
1 & 0 & \cdots & 0\\
\hline
0 &  & & \\
\vdots &  & h& 
\\
0 &  & &
\end{array}
\right).
\end{equation*} 
Moreover, for an element $s \in \Sph^{N-1}$, we denote by $V_{s}$ an element of $\SOR{N}$
such that $V_{s} e_{1} = s$. It is constructed in the following way. When $\langle s,e_{1} \rangle
\not =0$, the family  $(s,e_{2},\dots,e_{N})$ is free. We can orthonormalize it by means 
of the Gramm-Schmidt procedure. By induction, we let
\begin{equation}
\label{eq:GS}
\begin{split}
u_{1}:=s, \quad u_i :=e_i - \sum_{k=1}^{i-1}\langle e_i,u_k\rangle \frac{u_k}{\vert u_{k}\vert^2},\quad i\in\{ 2,\cdots,N\},
\end{split}
\end{equation}
and then $s_i:=u_i/|u_i|$, for all $i \in \{1,\dots,N\}$, so that $s_{1}=s$. 
Then, the family $(s_{1},s_{2},\dots,s_{N})$ is an orthonormal basis and $V_{s}$ is given by 
the passage matrix expressing the $(s_{i})_{1 \leq i \leq N}$' in the basis $(e_{i})_{1 \leq i \leq N}$. 
When $\langle s,e_{1} \rangle =0$, we consider $\langle s,e_{2} \rangle$. If $\langle s,e_{2} \rangle \not =0$, then 
the family $(s,e_{3},\dots,e_{N},e_{1})$ is free and we can apply the Gramm-Schmidt procedure. 
If $\langle s,e_{2} \rangle =0$, 
we then go on until we find some index $k \in \{3,\dots,N\}$
such that $\langle s,e_{k} \rangle \not = 0$. Such a construction ensures that 
the mapping $\Sph^{N-1} \ni s \mapsto V_{s} \in \SOR{N}$ is measurable. 

With $s \mapsto V_{s}$ and $h \mapsto L_{h}$ at hand, we claim that the mapping 
$\phi : (s,h) \mapsto V_{s} L_{h}$ is bijective from $\Sph^{N-1} \times \SOR{N-1}$ onto
$\SOR{N}$. Given some $g \in \SOR{N}$, $g=V_{s}L_{h}$ if and only if 
$s=ge_{1}$ and $L_{h}=V_{ge_{1}}^\top g$. By construction of $V_s$ and orthogonality of $V_{ge_{1}}^\top g$,
we indeed check that $(V_{ge_{1}}^\top g)_{1,1}=1$ and $(V_{ge_{1}}^\top g)_{i,1}=(V_{ge_{1}}^\top g)_{1,i}=0$ for $i=2,\dots,N$.
In other words, 
$V_{ge_{1}}^\top g$ always fits some $L_{h}$, the value of $h$ being uniquely determined by 
the lower block $(V_{ge_{1}}^\top g)_{2 \leq i,j \leq N}$, which proves the bijective property of $\phi$. 
Denoting by $\Pi_{N-1}$ the projection mapping:
\begin{equation*}
\pi_{N-1} :  
{\mathcal M}_{N}(\R)  \ni (a_{i,j})_{1 \leq i,j \leq N} \mapsto 
(a_{i,j})_{2 \leq i,j \leq N},
\end{equation*}
we deduce that the converse of $\phi$ writes
$\phi^{-1} : \SOR{N} \ni g \mapsto (ge_{1},\pi_{N-1}(V_{ge_{1}}^\top g)) \in \Sph^{N-1}
\times \SOR{N-1}$.

The mapping $\phi$ allows us to disintegrate the Haar measure on $\SOR{N}$ in terms of the product of 
the Lebesgue probability measure $\nu_{\Sph^{N-1}}$ on the sphere $\Sph^{N-1}$ and the Haar probability measure on 
$\SOR{N-1}$. We have the following result, see e.g. Proposition III.3.2 in \cite{fran:leja:12} for a proof:
\begin{lemma}[Representation of the Haar measure on $\SOR{N} $]
Let $f$ be a bounded Borel function from $\SOR{N}$ to $\R$. Then
(with $\nu_{\Sph^{N-1}}$ the normalized Lebesgue measure on $\Sph^{N-1}$),
\begin{eqnarray}
\label{DESINT_HAAR}
\int_{\SOR{N}}f(g)d\mu_{\SO{N}}(g):=\int_{\Sph^{N-1}\times \SOR{N-1}} f(V_s L_{h}) d\nu_{\Sph^{N-1}}(s)
d\mu_{\SO{N-1}}(h).
\end{eqnarray}
\end{lemma}

\subsection{Proof of Theorem \ref{bds_EXPL_UE}}
From \eqref{STROOCK_SALOFF} and  Theorem \ref{bds}, we derive the following two-sided bound for the conditional density. There exists $\tilde C:=\tilde C(N)\ge 1$ such that, for all $t>0 $,
\begin{align}
&\frac{1}{\tilde C\beta [t^{N/2}(1\wedge t)^{N(N-1)/4}]}\int_{\SOR{N}}\exp \Bigl( - \Bigl\{ \beta \frac{d^2_{\SO{N}}(\textrm{Id}_{N},g)}{t}+\tilde C\frac{ \vert v - g x_{0} \vert^2}t\Bigr \}\Bigr) d \mu_{\SO{N}}(g)
\nonumber
\\
&\hspace{10pt}\le 
f_{x_{0}}(t,v)
\label{bds_SG_SON}
\\
&\hspace{20pt}
\le \frac{\tilde C\beta}{t^{N/2}(1\wedge t)^{N(N-1)/4}}\int_{\SO{N}(\R)}\exp \Bigl( -  \Bigl\{\frac{d^2_{\SO{N}}(\textrm{Id}_{N},g)}{\beta t} 
+\frac{ \vert v - g x_{0} \vert^2}{\tilde Ct}\Bigr\} \Bigr) d \mu_{\SO{N}}(g),
\nonumber
\end{align}
which will be the starting point to derive the bounds of Theorem \ref{bds_EXPL_UE}.
\color{black}

\subsubsection{Gaussian Regime}
Let us first concentrate on the bounds when $|x_0|\wedge |v|\le 1$. 
Without loss of generality, 
we can assume by symmetry that $|x_0|\le 1$. 
Indeed, for all $g\in \SOR{N}$, $|v-gx_0|=|g^\top v-x_0|$
and $d_{\SO{N}}(\textrm{Id}_{N},g)
=d_{\SO{N}}(\textrm{Id}_{N},g^\top)$. Moreover, the Haar measure
 is invariant by transposition. This can be checked as follows.
If $Z$ is distributed according to the Haar measure, then, for any rotation $\rho$, $\rho Z^{\top} = (Z \rho^{\top})^{\top}$. Since 
$Z \rho^{\top}$ has the same law as $Z$ (as the group is compact, it is known the Haar measure is invariant both by left and right multiplications), we deduce that the law of $Z^{\top}$ is invariant by rotation. 

Now, write:
\begin{equation*}
\frac 12\frac{|x_0-v|^2}{t}-\frac{\|\textrm{Id}_{N}-g\|^2|x_0|^2}{t}  \le \frac{|v-gx_0|^2}t\le 2 
\biggl(\frac{\|\textrm{Id}_{N}-g\|^2|x_0|^2}t+\frac{|x_0-v|^2}t\biggr).
\end{equation*}
From \eqref{bds_SG_SON} and the assumption $\vert x_{0}\vert \leq 1$, we get that: 
\begin{equation*}
\begin{split}
f_{x_{0}}(t,v)& \ge \frac{(\tilde C\beta)^{-1}}{t^{N/2}}\exp	\bigl(-2\tilde C\frac{|x_0-v|^2}{t} \bigr)\\%
&\times\left\{ \frac{1}{(1\wedge t)^{N(N-1)/4} }\int_{\SOR{N}}\exp\Bigl(-\beta\frac{d^2_{\SO{N}}({\rm Id}_N,g) }{t}- 2\tilde C\frac{\|{\rm Id}_N-g\|^2}{t} \Bigr)d\mu_{\SO{N}}(g) \right\}\\
&\hspace{-16pt} \overset{\rm Lemma \ \ref{lem:carnot}}{\ge }
 \frac{\tilde C^{-1}}{t^{N/2}}\exp \bigl(-2\tilde C\frac{|x_0-v|^2}{t} \bigr)\\%
&\hspace{15pt}\times\left\{ \frac{1}{(1\wedge t)^{N(N-1)/4} }\int_{\SOR{N}}
\exp \Bigl(-\tilde C\frac{d^2_{\SO{N}}({\rm Id}_N,g) }{t} \Bigr)d\mu_{\SO{N}}(g) \right\}
\\
&\hspace{-5pt} \overset{\eqref{STROOCK_SALOFF}}{\ge} \frac{\tilde C^{-1}}{t^{N/2}}\exp \bigl(-2\tilde 
C\frac{|x_0-v|^2}{t} \bigr),
\end{split}
\end{equation*}
the constant $\tilde C$ being allowed to increase from line to line. 
On the other hand,
using once again Lemma \ref{lem:carnot} and \eqref{STROOCK_SALOFF} and
 choosing $\tilde C$ large enough:
\begin{equation*}
\begin{split}
f_{x_{0}}(t,v)& \le \frac{\tilde C\beta}{t^{N/2}}\exp\biggl(-\frac{\tilde C^{-1}}2\frac{|x_0-v|^2}{t}\biggr)\\
&\hspace{15pt} \times\left\{ \frac{1}{(1\wedge t)^{N(N-1)/4} }\int_{\SOR{N}}\exp
\Bigl(-\frac{d^2_{\SO{N}}({\rm Id}_N,g) }{\beta t}+\frac{\|{\rm Id}_N-g\|^2}{\tilde Ct} \Bigr)d\mu_{\SO{N}}(g) \right\}\\
&\le \frac{\tilde C}{t^{N/2}}\exp\biggl(- \frac{|x_0-v|^2}{\tilde C t}\biggr),
\end{split}
\end{equation*}
where we have chosen $\tilde C$ such that, for all $g\in \SOR{N}$,
\color{black}
\begin{equation*}
\exp\biggl(\biggl\{-\frac{d_{\SO{N}}^2(\textrm{Id}_{N},g)}{\beta t}+\frac{\|\textrm{Id}_{N}-g\|^2}{\tilde C t} \biggr\}\biggr)\le \exp\biggl(-\frac{d_{\SO{N}}^2(\textrm{Id}_{N},g)}{2\beta t}  \biggr).
\end{equation*}

\subsubsection{Non Gaussian Regime}
We now look at the case 
$|x_0|\wedge |v|> 1$.
Starting from \eqref{bds_SG_SON} and Lemma \ref{lem:carnot},
 we aim at giving, for given $c>0$, upper and lower bounds, homogeneous 
 to those of \eqref{LOSS_MASS_UE}, for the quantity $t^{-N/2}p_{x_{0}}(t,v)$, where :
\begin{equation}
\label{eq:8:12:1}
p_{x_{0}}(t,v) := (1\wedge t)^{-N(N-1)/4} \int_{{\mathrm{SO}}_{N}(\R)} \exp \Bigl( - c 
\frac{\| \textrm{Id}_{N}- g \|^2 + \vert v - gx_{0} \vert^2}t \Bigr) d \mu_{\SO{N}}(g).
\end{equation}
Above, we notice that $\vert v - gx_{0} \vert^2 = \vert g^{\top} v - x_{0} \vert^2$ and 
$\| \textrm{Id}_{N}- g \|^2 = \| \textrm{Id}_{N}- g^{\top} \|^2$. Since 
the Haar measure is invariant by transposition,
the roles of $v$ and $x_{0}$ can be exchanged in formula \eqref{eq:8:12:1} 
and we can assume that $\vert v \vert \geq \vert x_{0}\vert$.\\ 

By Lemma \ref{proj} (with $x_{0}$ replaced by $g x_{0}$), we know that  
\begin{equation*}
\bigl\vert  \vert v \vert - \vert x_{0} \vert \bigr\vert^2+ 
\bigl\vert \frac{\vert x_0\vert}{\vert v \vert}
v - g x_{0} 	\bigr\vert^2 \leq \vert v - g x_{0} \vert^2 \leq
 2\Bigl( \bigl\vert  \vert v \vert - \vert x_{0} \vert \bigr\vert^2 + 
 \bigl\vert \frac{\vert x_{0}\vert}{\vert v \vert}
v - g x_{0} 	\bigr\vert^2 \Bigr). 
\end{equation*}

\textit{Radial cost.}
 The term $\vert  \, \vert v \vert - \vert x_{0} \vert \, \vert^2$ is referred to as the radial cost. Since it is independent of
$g$, we can focus on the other one, called the tangential cost. 
Then, changing $v$ into 
$(\vert x_{0} \vert/ \vert v \vert) v$, 
we can assume that $\vert v \vert = \vert x_{0} \vert$.

\textit{Tangential cost.}
We now assume that $\vert v \vert = \vert x_{0} \vert$. By rotation, we can assume that $x_{0}= \vert x_{0} \vert e_{1}$. Then, we can write
$v= \vert x_{0} \vert h e_{1}$ for some $h \in \textrm{SO}_{N}(\R)$. We then expand in \eqref{eq:8:12:1}
\begin{equation*}
\| \textrm{Id}_{N} - g \|^2 + \vert x_0 \vert^2\vert h e_{1} - ge_{1} \vert^2 = \vert e_{1} - g e_{1} \vert^2 + 
\vert x_{0} \vert^2 \vert h e_{1} - ge_{1} \vert^2 + \sum_{i=2}^N \vert e_{i} - g e_{i} \vert^2.
\end{equation*}
The strategy is then quite standard and consists in reducing the quadratic form 
$\vert e_{1} - g e_{1} \vert^2 + 
\vert x_{0} \vert^2 \vert h e_{1} - ge_{1} \vert^2$. We write
\begin{equation*}
\begin{split}
&\vert e_{1} - g e_{1} \vert^2 + 
\vert x_{0} \vert^2 \vert h e_{1} - ge_{1} \vert^2
\\
&= \bigl( 1+ \vert x_{0} \vert^2 \bigr) \vert g e_{1} \vert^2
- 2 \langle ge_{1}, e_{1} + \vert x_{0} \vert^2 h e_{1} \rangle + 1+ \vert x_{0} \vert^2
\\
&= \bigl( 1+ \vert x_{0} \vert^2 \bigr) \Bigl\vert g e_{1} - \frac{ e_{1} 
+ \vert x_{0} \vert^2 h e_{1}}{1+ \vert x_{0} \vert^2}
\Bigr\vert^2 - \frac{1}{1+ \vert x_{0} \vert^2} \bigl\vert e_{1} + \vert x_{0} \vert^2 h e_{1} \bigr\vert^2
+ 1+ \vert x_{0} \vert^2. 
\end{split}
\end{equation*}
Since,
\begin{equation*}
\begin{split}
\frac{1}{1+ \vert x_{0} \vert^2} \bigl\vert e_{1} + \vert x_{0} \vert^2 h e_{1} \bigr\vert^2
- \bigl( 1+ \vert x_{0} \vert^2 \bigr)
&=\frac{1}{1+ \vert x_{0} \vert^2}
\Bigl( \bigl\vert e_{1} + \vert x_{0} \vert^2 h e_{1} \bigr\vert^2 -  
\bigl( 1+ \vert x_{0} \vert^2 \bigr)^2 \Bigr)
\\
&=- \frac{2 \vert x_{0} \vert^2}{1+ \vert x_{0} \vert^2} \bigl( 1 -   \langle e_{1}, h e_{1} \rangle 
\bigr),
\end{split}
\end{equation*}
we finally get that 
\begin{equation*}
\begin{split}
&\vert e_{1} - g e_{1} \vert^2 + 
\vert x_{0} \vert^2 \vert h e_{1} - ge_{1} \vert^2
\\
&= \bigl( 1+ \vert x_{0} \vert^2 \bigr) \Bigl\vert g e_{1} - \frac{ e_{1} 
+ \vert x_{0} \vert^2 h e_{1}}{1+ \vert x_{0} \vert^2}
\Bigr\vert^2 + \frac{2 \vert x_{0} \vert^2}{1+ \vert x_{0} \vert^2} \bigl( 1 -   \langle e_{1}, h e_{1} \rangle 
\bigr). 
\end{split}
\end{equation*}
As the second term is independent of $g$, we write
\begin{equation}
\label{CTR_PX0}
\begin{split}
&p_{x_0}(t,v) =(1\wedge  t)^{-N(N-1)/4} \exp \biggl( - \frac{c}{t}
\frac{2 \vert x_{0} \vert^2}{1+ \vert x_{0} \vert^2} \bigl( 1 -   \langle e_{1}, h e_{1} \rangle 
\bigr) \biggr)
\\
&\hspace{15pt} \times
\int_{{\mathrm{SO}}_{N}(\R)} \exp \biggl(  - \frac{c}{t}
\bigl( 1+ \vert x_{0} \vert^2 \bigr) \Bigl\vert g e_{1} - \frac{ e_{1} 
+ \vert x_{0} \vert^2 h e_{1}}{1+ \vert x_{0} \vert^2}
\Bigr\vert^2 - \frac{c}{t} \sum_{i=2}^N \vert e_{i} - g e_{i} \vert^2
\biggr) d \mu_{\SO{N}}(g).
\end{split}
\end{equation}
Now, we notice that 
\begin{equation*}
\Bigl\vert \frac{ e_{1} 
+ \vert x_{0} \vert^2 h e_{1}}{1+ \vert x_{0} \vert^2}
\Bigr\vert \leq 1.
\end{equation*}
Since $\vert x_{0} \vert^2>1$, we have  $\vert e_{1} 
+ \vert x_{0} \vert^2 h e_{1} \vert >0$. Therefore, we can proceed as in 
the previous paragraph: in the first term inside the second exponential in 
\eqref{CTR_PX0}, we use Remark \ref{proj2}
 to split the radial and tangential costs. The radial cost is here given by 
\begin{equation*}
\begin{split}
\Bigl(1 - \Bigl\vert \frac{ e_{1} 
+ \vert x_{0} \vert^2 h e_{1}}{1+ \vert x_{0} \vert^2}
\Bigr\vert \Bigr)^2
&= \Bigl(1 - \Bigl\vert \frac{ e_{1} 
+ \vert x_{0} \vert^2 h e_{1}}{1+ \vert x_{0} \vert^2}
\Bigr\vert^2 \Bigr)^2 \Bigl(1 + \Bigl\vert \frac{ e_{1} 
+ \vert x_{0} \vert^2 h e_{1}}{1+ \vert x_{0} \vert^2}
\Bigr\vert \Bigr)^{-2}
\\
&= \Bigl[1 - \Bigl( 1 -  \frac{2\vert x_{0} \vert^2}{(1+ \vert x_{0} \vert^2)^2}
\bigl( 1 - \langle e_{1} , h e_{1} \rangle \bigr)
 \Bigr)
\Bigr]^2 \Bigl(1 + \Bigl\vert \frac{ e_{1} 
+ \vert x_{0} \vert^2 h e_{1}}{1+ \vert x_{0} \vert^2}
\Bigr\vert \Bigr)^{-2}
\\
&= \bigl[ \frac{2\vert x_{0} \vert^2}{(1+ \vert x_{0} \vert^2)^2} \bigr]^2
\bigl( 1 - \langle e_{1} , h e_{1} \rangle \bigr)^2
\Bigl(1 + \Bigl\vert \frac{ e_{1} 
+ \vert x_{0} \vert^2 h e_{1}}{1+ \vert x_{0} \vert^2}
\Bigr\vert \Bigr)^{-2}.
\end{split}
\end{equation*}
Up to multiplicative constants, the last term above can be bounded from above
by $(1+ \vert x_{0} \vert^2)^{-2}(1 - \langle e_{1},h e_{1}\rangle)$. 
In particular, up to a modification of the constant $c$ in $p_{x_{0}}(t,v)$, 
we can see the radial cost as a part of the exponential pre-factor in \eqref{CTR_PX0}.
Therefore,
without any ambiguity, we can slightly modify the definition of $
p_{x_{0}}(t,v)$ and assume 
that it writes 
\begin{equation}
\label{FORME_AVEC_COUT}
\begin{split}
&p_{x_{0}}(t,v) = (1\wedge t)^{-N(N-1)/4} \exp \biggl( - \frac{c}{t}
\frac{2 \vert x_{0} \vert^2}{1+ \vert x_{0} \vert^2} \bigl( 1 -   \langle e_{1}, h e_{1} \rangle 
\bigr) \biggr)
\\
&\hspace{15pt} \times
\int_{{\mathrm{SO}}_{N}(\R)} \exp \biggl(  - \frac{c}{t}
\bigl( 1+ \vert x_{0} \vert^2 \bigr) \Bigl\vert g e_{1} - \frac{ e_{1} 
+ \vert x_{0} \vert^2 h e_{1}}{\vert e_{1} + \vert x_{0} \vert^2 h e_{1} \vert }
\Bigr\vert^2  - \sum_{i=2}^N \frac{\vert e_{i} - g e_{i} \vert^2}t \biggr) d \mu_{\SO{N}}(g).
\end{split}
\end{equation}

Equation \eqref{DESINT_HAAR} now yields:
\begin{equation}
\label{APRES_CH_VAR}
\begin{split}
&p_{x_{0}}(t,v)=\frac{1}{(1\wedge t)^{N(N-1)/4}} \exp \biggl( - \frac{c}{t}
\frac{2 \vert x_{0} \vert^2}{1+ \vert x_{0} \vert^2} \bigl( 1 -   \langle e_{1}, h e_{1} \rangle 
\bigr) \biggr)
\\
&\hspace{1pt} \times
\int \exp \biggl(  - \frac{c}{t}\biggr\{
\bigl( 1+ \vert x_{0} \vert^2 \bigr) \bigl\vert s - \bar{s}\bigr\vert^2  + \sum_{i=2}^N \vert e_{i} - 
V_s L_{k} e_{i}\vert^2\biggr\}\biggr)
d \nu_{\Sph^{N-1}}(s) d\mu_{\SO{N-1}}(k),
\end{split}
\end{equation}
the integral being defined on $\Sph^{N-1}\times \SOR{N-1}$,
with 
\begin{equation}
\label{BAR_S}
\bar{s} = ( e_{1} 
+ \vert x_{0} \vert^2 h e_{1})/(\vert e_{1} + \vert x_{0} \vert^2 h e_{1} \vert).
\end{equation}
\vspace{5pt}

\textit{Lower bound.}
Observe first that, for all $i\in \{2,\cdots,N\}$, 
$|e_i-V_s L_{k} e_{i} |^2\le 2( |(\textrm{Id}_{N}-V_s)e_i|^2+|e_i-L_{k} e_{i}|^2)$, using that $V_s$ defines an isometry for the last control. From Lemma \ref{lem:carnot}, we now derive that 
$\sum_{i=2}^N|e_i-L_{k} e_{i}|^2 \le c_1 \|{\rm Id}_{N-1}-k\|^2\le c_2 d_{\SO{N-1}}^2({\rm Id}_{N-1},k)$, where $(c_1,c_2):=(c_1,c_2)(N)$. By \eqref{STROOCK_SALOFF}, applied for $N-1$, we get that 
there exists $C:=C(N)\ge 1$ (the value of which is allowed to increase below) such that
\begin{equation*}
\frac{1}{(1\wedge t)^{(N-1)(N-2)/4}}\int_{\SOR{N-1}}\exp\left(-2 c c_2\frac{d_{\SO{N-1}}^2({\rm Id}_{N-1},k)}{t} \right)
d\mu_{\SO{N-1}}(k)\ge C^{-1}.
\end{equation*}
Thus,
\begin{equation*}
\begin{split}
p_{x_0}(t,v)&\ge \frac{1}{C (1\wedge t)^{(N-1)/2}} \exp \biggl( - \frac{c}{t}
\frac{2 \vert x_{0} \vert^2}{1+ \vert x_{0} \vert^2} \bigl( 1 -   \langle e_{1}, h e_{1} \rangle 
\bigr) \biggr) \\
&\hspace{15pt}\times
\int_{\Sph^{N-1}} \exp \biggl(  - \frac{c}{t}\biggr\{
\bigl( 1+ \vert x_{0} \vert^2 \bigr) \bigl\vert s - \bar{s}
\bigr\vert^2  + 2\sum_{i=2}^N \vert e_{i} - V_se_i \vert^2 \biggr\}\biggr)
d \nu_{\Sph^{N-1}}(s).
\end{split}
\end{equation*}
Let us restrict the integral to a neighborhood of $\bar s$ in $\Sph^{N-1} $
of the form 
\begin{equation}
\label{VOIS}
{\mathcal V}_{\bar s} :=\{ s : \exists \ R\in \SOR{N},\ s= R \bar{s},\ \|R-{\rm Id}_N\|\le t^{1/2}/|x_0| \}. 
\end{equation}
Then,
\begin{equation}
\begin{split}
p_{x_0}(t,v) &\ge \frac{1}{C (1\wedge t)^{(N-1)/2}} \exp \biggl( - \frac{c}{t}
\frac{2 \vert x_{0} \vert^2}{1+ \vert x_{0} \vert^2} \bigl( 1 -   \langle e_{1}, h e_{1} \rangle 
\bigr) \biggr)
\\
&\hspace{15pt} \times
\int_{{\mathcal V}_{\bar s}} \exp \biggl(  - \frac{2c}{t} \sum_{i=2}^N \vert e_{i} - V_s e_i \vert^2 \biggr\}\biggr)
 d\nu_{\Sph^{N-1}}(s).
\end{split}
\label{PREAL_MINO}
\end{equation}
As the set of the $s$'s such that $\langle s,e_{1} \rangle =0$ is of zero measure, we can restrict the integral to the set of 
 $s\in \Sph^{N-1}$ such that $\langle s,e_{1} \rangle \not = 0$. 
 By construction (see \eqref{eq:GS}), 
 $V_se_{i}=s_{i}$, for $i \in \{1,\dots,N\}$, with $s_1=s$ and
\begin{equation}
\label{eq:GS:2}
u_i =e_i-\sum_{k=1}^{i-1}\langle e_i,s_k\rangle s_k,\ s_i:=\frac{u_i}{|u_i|},\ i\in\{ 2,\cdots,N\},\ u_1=s.
\end{equation}
We can write, for $i \in \{2,\dots, N\}$, 
\begin{equation}
\label{eq:GS:3}
\begin{split}
\vert e_{i} - V_se_i \vert^2
= \vert e_{i} - s_i \vert^2 &\le 2 \bigl(|e_i-u_i|^2+|u_i-s_i|^2\bigr)
\\
&= 2 \bigl(|e_i-u_i|^2+|1 - \vert u_i \vert |^2\bigr)
\\
&\leq 2 \bigl(|e_i-u_i|^2+|\vert e_{i}\vert - \vert u_i \vert |^2\bigr) 
\leq 4 \bigl(|e_i-u_i|^2\bigr).
\end{split}
\end{equation}
Now, by \eqref{eq:GS:2},
\begin{equation}
\label{eq:GS:4}
\vert e_i - u_i \vert^2 = 
\sum_{k=1}^{i-1}\langle e_i,s_{k}\rangle^2 
= 
 \langle e_i,s_{1} \pm e_{1} \rangle^2+
\sum_{k=2}^{i-1}\langle e_i,s_{k} - e_{k} \rangle^2
\leq 
 \vert s_{1} \pm e_{1} \vert^2 +
\sum_{k=2}^{i-1}\vert s_{k} - e_{k} \vert^2.
\end{equation}
Therefore, by \eqref{eq:GS:3} and \eqref{eq:GS:4} and by a standard induction,
for all $i \in \{2,\dots,N\}$,
\begin{equation}
\vert e_{i} - s_i \vert^2 \leq \bar C \vert e_{1} \pm s_1 \vert^2
= 2 \bar{C} \bigl(1 \pm \langle e_{1},s \rangle\bigr) 
= 2 \bar{C} \frac{1 - \langle e_{1},s \rangle^2}{1 \mp \langle e_{1},s \rangle}.
\end{equation}
In the above, we can always choose the sign in $\mp$ so that 
$1 \mp \langle e_{1},s \rangle \geq 1$. Therefore, for all $i \in \{2,\dots,N\}$,
\begin{equation}
\label{CTR_RESTES_MINO}
\vert e_{i} - s_i \vert^2 
\leq 2 \bar{C} \bigl( 1 - \langle e_{1},s \rangle^2 \bigr)
= 2 \bar{C} \sum_{k=2}^N \langle e_{k},s \rangle^2. 
\end{equation}
%
%
%
Since, for $s\in {\mathcal V}_{\bar s}$,  
$|\langle s,e_k\rangle|\le |\langle \bar s,e_k\rangle| + t^{1/2}/|x_0|$, we 
deduce from \eqref{CTR_RESTES_MINO}: 
\begin{equation*}
\begin{split}
\sum_{i=2}^N|e_i-s_{i}|^2&\le \bar C \biggl( \sum_{i=2}^N |\langle \bar s,e_i\rangle|^2+  (N-1)t/|x_0|^2\biggr)
\\
&\le \bar C\bigl( 1-\langle \bar s,e_1\rangle ^2 +t/|x_0|^2\bigr) \le \bar C\bigl( 
2(1-\langle \bar s,e_1\rangle)+t/|x_0|^2 \bigr).
\end{split}
\end{equation*}
We derive from  \eqref{PREAL_MINO} that 
\begin{equation*}
\begin{split}
p_{x_0}(t,v)&
\ge \frac{\nu_{\Sph^{N-1}}( {\mathcal V}_{\bar s})
}{\bar C (1\wedge t)^{(N-1)/2}} \exp \biggl( - \frac{\bar C}{t}\Bigl[
\frac{2 \vert x_{0} \vert^2}{1+ \vert x_{0} \vert^2} \bigl( 1 -   \langle e_{1}, h e_{1} \rangle 
\bigr) + \bigl(1-\langle e_1,\bar s\rangle \bigr)\Bigr]\biggr)
\\
&\ge \bar C^{-1} 
 \delta_t^{N-1} \exp \biggl( - \frac{\bar C}{t}\Bigl[
\frac{2 \vert x_{0} \vert^2}{1+ \vert x_{0} \vert^2} \bigl( 1 -   \langle e_{1}, h e_{1} \rangle 
\bigr) + \bigl(1-\langle e_1,\bar s\rangle \bigr)\Bigr]\biggr),
\end{split}
\end{equation*}
denoting, as in Theorem \ref{bds_EXPL_UE}, 
$\displaystyle \delta_{t}  :=  \frac{1\wedge 
 \Bigl( \frac{\displaystyle t^{1/2}}{\displaystyle  |x_0|} \Bigr)}{1\wedge t^{1/2}}$ and using \eqref{VOIS} for the last inequality.
 
Assume first that $\langle e_1,he_1\rangle \le 0  $. The above equation yields 
\begin{equation*}
\begin{split}
p_{x_0}(t,v)& \ge \bar C^{-1}\delta_t^{N-1} \exp \bigl( - \frac{\bar C}{t}\bigr)\ge \bar C^{-1}\delta_t^{N-1} 
\exp \Bigl( - \frac{\bar C}{t} \bigl(1-\langle e_1,he_1\rangle \bigr)\Bigr). 
\end{split}
\end{equation*}
Recalling that, for the tangential cost analysis, we have assumed $|x_0|=|v| $, we derive 
\begin{equation*}
1-\langle e_1,he_1\rangle=1-\langle \frac{x_0}{|x_0|}, \frac{v}{|x_0|}\rangle =\frac1{2|x_0|^2}|x_0-v|^2,
\end{equation*}
 which gives the claim. 
 
 Assume now that $\langle e_1,he_1\rangle \ge 0  $. It can be checked from the definition of $\bar s $ in \eqref{BAR_S} that $\langle e_1,\bar s \rangle\ge \langle e_1,he_1 \rangle  $ so that we eventually get:
\begin{equation*}
p_{x_0}(t,v)\ge \bar C^{-1}\delta_t^{N-1}
\exp \Bigl( - \frac{\bar C}{t} \bigl(1-\langle e_1,he_1\rangle \bigr)\Bigr).
\end{equation*}
We conclude by the same argument as above. \\
\color{black}

\textit{Upper bound.} 
Going back to \eqref{APRES_CH_VAR} and using the fact that $V_{s} \in 
\SOR{N}$ for any $s \in \Sph^{N-1}$, we get
\begin{equation}
\label{eq:nouvelle:BSUP}
\begin{split}
&p_{x_{0}}(t,v) =\frac{1}{(1\wedge t)^{N(N-1)/4}} \exp \biggl( - \frac{c}{t}
\frac{2 \vert x_{0} \vert^2}{1+ \vert x_{0} \vert^2} \bigl( 1 -   \langle e_{1}, h e_{1} \rangle 
\bigr) \biggr)
\\
&\hspace{1pt} \times
\int \exp \biggl(  - \frac{c}{t}\biggr\{
\bigl( 1+ \vert x_{0} \vert^2 \bigr) \bigl\vert s - \bar{s}\bigr\vert^2  + 
\sum_{i=2}^N \vert \hat s_{i} - L_{k} e_{i}\vert^2\biggr\}\biggr)
d \nu_{\Sph^{N-1}}(s) d\mu_{\SO{N-1}}(k), 
\end{split}
\end{equation}
where $\hat s_i=V_s^\top e_i,\ i\in \{ 2,\cdots,N\} $.
We then focus on the integral with respect to $k$, namely
\begin{equation*}
q_t(s):=
(1\wedge t)^{-(N-1)(N-2)/4}\int_{\SOR{N-1}} \exp \Bigl(  - \frac{c}{t}
 \sum_{i=2}^N \vert \hat s_{i} - L_{k} e_{i}\vert^2 \Bigr)
 d\mu_{\SO{N-1}}(k), 
\end{equation*}
for a given $s \in \Sph^{N-1}$, the normalization $(1\wedge t)^{(N-1)(N-2)/4}$ standing 
for the order of the volume of the ball of radius $t^{1/2}$ in $\SOR{N-1}$.
Denoting by $\hat s^{2,N}$ the $N-1$ square matrix 
made of the column vectors $((\hat s_{2})_{j})_{2 \leq j \leq N}$, \dots,
$((\hat s_{N})_{j})_{2 \leq j \leq N}$, where 
$(\hat s_{i})_{j}$ stands for the $j$th coordinate of $\hat s_{i}$, we get
\begin{equation*}
\sum_{i=2}^N \vert \hat s_{i} - L_{k} e_{i}\vert^2 \geq \| \hat s^{2,N} - k \|^2. 
\end{equation*} 

Now, we distinguish two cases. For a given $\varepsilon >0$ 
to be specified next, 
we first consider the case when $\|\hat s^{2,N}-k\| \geq \varepsilon$ for any $k \in \SOR{N-1}$. Then, there 
exists a constant $c':=c'(\varepsilon)>0$ such that 
$\|\hat s^{2,N}-k\| \geq c' d_{\SO{N-1}}({\rm Id}_{N-1},k)$, so that (up to a modification of $c$)
\begin{equation*}
q_t(s) \leq (1\wedge t)^{-(N-1)(N-2)/4}
\int_{\SOR{N-1}} \exp \Bigl(  - \frac{c}{t} d^2_{\SO{N-1}}({\rm Id}_{N-1},k) \Bigr)
 d\mu_{\SO{N-1}}(k) \leq \bar C,
\end{equation*}
for a constant $\bar C:=\bar C(N)$. 

Let us now assume that there exists $k_{0} \in \SOR{N-1}$ such that  $\|\hat s^{2,N}-k_{0}\| \leq \varepsilon$. 
By invariance by rotation of the Haar measure, we notice that $q_{t}(s)$ can be bounded by 
\begin{equation*}
\begin{split}
q_t(s) &\leq 
(1\wedge t)^{-(N-1)(N-2)/4}\int_{\SOR{N-1}} \exp \Bigl(  - \frac{c}{t}
 \| \hat s^{2,N} - k_{0} k \|^2 \Bigr)
 d\mu_{\SO{N-1}}(k)
\\
&= 
(1\wedge t)^{-(N-1)(N-2)/4}\int_{\SOR{N-1}} \exp \Bigl(  - \frac{c}{t}
 \| k_{0}^{\top}\hat s^{2,N} - k \|^2 \Bigr)
 d\mu_{\SO{N-1}}(k). 
 \end{split}
\end{equation*}
Letting $\tilde{s}^{2,N} := k_{0}^{\top} \hat s^{2,N}$, we notice that 
$\|\tilde{s}^{2,N}-{\rm Id}_{N-1}\| \leq \varepsilon$.  
This permits to define $\tilde{S}^{2,N} := \ln(\tilde{s}^{2,N})$ (provided $\varepsilon$
is chosen small enough).   
 
Again, we distinguish two cases, according to the value of the variable $k$ in the integral. 
When $\|\tilde{s}^{2,N} - k\| \geq \varepsilon$, we can use the same trick as before and say 
that  $\|\tilde{s}^{2,N} - k\| \geq c d_{\SO{N}}({\rm Id}_{N-1},k)$. Repeating the computations, 
we get 
\begin{equation*}
(1\wedge t)^{-(N-1)(N-2)/4}\int_{
\|\tilde{s}^{2,N} - k\| \geq \varepsilon} \exp \Bigl(  - \frac{c}{t}
 \| \tilde{s}^{2,N} - k \|^2 \Bigr)
 d\mu_{\SO{N-1}}(k) \leq \bar C.
 \end{equation*}
When $\|\tilde{s}^{2,N} - k\| \leq \varepsilon$, we have 
$\|{\rm Id}_{N-1} - k \| \leq 2 \varepsilon$, so that 
$k$ may be inverted by the logarithm and written as $k=\exp(K)$
for some antisymmetric matrix $K$ of size $N-1$. 
By local Lipschitz property of the logarithm, we deduce that, for such a $k$ (and for a new value of $c'$),
\begin{equation*} 
 \|\tilde{s}^{2,N} - k\| \geq c' \|\tilde{S}^{2,N} - K\|. 
\end{equation*}
We then denote $\tilde{H}^{2,N}$ the orthogonal projection of $\tilde{S}^{2,N}$
on ${\mathcal A}_{N-1}(\R)$. We get 
\begin{equation*} 
 \|\tilde{s}^{2,N} - k\| \geq c' \|\tilde{H}^{2,N} - K\|. 
\end{equation*}
Clearly, $\tilde{H}^{2,N}$ is in the neighborhood of $0$. By local Lipschitz property of the exponential, 
we finally obtain (again, for a new value of $c'$)
\begin{equation*} 
 \|\tilde{s}^{2,N} - k\| \geq c' \|\exp(\tilde{H}^{2,N}) - k\|. 
\end{equation*}
Letting $\tilde{h}^{2,N} : =\exp(\tilde{H}^{2,N})$, we end up with
\begin{equation*}
\begin{split}
&(1\wedge t)^{-(N-1)(N-2)/4}\int_{
\|\tilde{s}^{2,N} - k\| \leq \varepsilon} \exp \Bigl(  - \frac{c}{t} 
 \| \tilde{s}^{2,N} - k \|^2 \Bigr)
 d\mu_{\SO{N-1}}(k) 
\\
&\leq (1\wedge t)^{-(N-1)(N-2)/4}\int_{
\|\tilde{s}^{2,N} - k\| \leq \varepsilon} \exp \Bigl(  - \frac{c}{t} 
 d^2_{\SO{N-1}}(\tilde{h}^{2,N},k) \Bigr)
 d\mu_{\SO{N-1}}(k), 
\end{split}
 \end{equation*}
 where we have used Lemma \ref{lem:carnot} on $\SOR{N-1}$ to get the second line. 
By a new rotation argument, 
\begin{equation*}
 (1\wedge t)^{-(N-1)(N-2)/4}\int_{
\|\tilde{s}^{2,N} - k\| \leq \varepsilon} \exp \Bigl(  - \frac{c}{t} 
 d^2_{\SO{N-1}}(\tilde{h}^{2,N},k) \Bigr)
 d\mu_{\SO{N-1}}(k) \leq \bar C,
 \end{equation*}
 which shows that $q_{t}(s) \leq \bar C$. 

Equation \eqref{eq:nouvelle:BSUP} thus yields:
\begin{equation*}
\begin{split}
p_{x_{0}}(t,v)&\le \frac{\bar C}{(1\wedge t)^{(N-1)/2}} \exp \biggl( - \frac{c}{t}
\frac{2 \vert x_{0} \vert^2}{1+ \vert x_{0} \vert^2} \bigl( 1 -   \langle e_{1}, h e_{1} \rangle 
\bigr) \biggr)
\\
&\hspace{15pt} \times
\int_{\Sph^{N-1}}\exp\biggl(-\frac{c(1+|x_0|^2)}{t}|s-\bar s|^2\biggr)d\nu_{\Sph^{N-1}}(s).
\end{split}
\end{equation*}
Observing now that there exists $\bar c>1$ such that 
$\bar c^{-1}|s-\bar s|\le d(s,\bar s) \le \bar c|s-\bar s|$, where $d$ stands for the Riemannian metric on the
sphere 
$\Sph^{N-1} $, we then deduce from the heat kernel estimates in 
Stroock \cite{stro:06} that 
$$
\frac1{1 \wedge  \Bigl(  \frac{\displaystyle t^{1/2}}{\displaystyle 1 + \vert x_{0} \vert} \Bigr)^{N-1}}
\exp\biggl(-\frac{c(1+|x_0|^2)}{t} 
\vert s - \bar s \vert^2
\biggr) \le  \bar C p_{\Sph^{N-1}}\left(\frac{t}{1+|x_0|^2} ,s,\bar s\right), $$
 where $p_{\Sph^{N-1}}$ stands for the heat kernel on $\Sph^{N-1}$.  
Since we have assumed  $|x_0|\ge 1$ we finally derive up to a modification of $\bar C$:
\begin{equation*}
p_{x_{0}}(t,v) \le \bar C\delta_t^{N-1} \exp \biggl( - \frac{c}{t}
\frac{2 \vert x_{0} \vert^2}{1+ \vert x_{0} \vert^2} \bigl( 1 -   \langle e_{1}, h e_{1} \rangle 
\bigr) \biggr) 
\le \bar C \delta_t^{N-1} \exp \Bigl( - 
\frac{1 -   \langle e_{1}, h e_{1} \rangle }{\bar C t} 
 \Bigr),
\end{equation*}
which gives an upper bound homogeneous to the lower bound and completes the proof.

\section{The degenerate case}
\label{SEC_DEG}

The strategy to complete the proof of Theorem \ref{thm:bd:degenerate} 
relies on an expansion of $Z_{t }$ in terms of iterated integrals of the Brownian motion on the  Lie algebra ${\mathcal A}_{N}(\R) $ of $\SOR{N}$. In that framework, it is worth mentioning that we do not exploit anymore the underlying group structure.  Instead, we explicitly make use of the \textit{Euclidean} structure of ${\mathcal A}_{N}(\R) $. Indeed the analysis relies on precise controls of events described by the whole trajectory of $Z$. We manage to handle the probability of those events by controlling the corresponding trajectories of  the ${\mathcal A}_{N}(\R)$-valued Brownian motion $B$. In that perspective, the heat kernel estimates \eqref{STROOCK_SALOFF} for the marginals of $Z$
in $\SOR{N}$ are not sufficient, as once again, the distribution of the whole path is needed to carry on the analysis.

\subsection{Set-up}
In the whole section, we will assume that degeneracy occurs along the first direction of the space, that is $X_{0}$ has the form:
\begin{equation*}
X_{0} = X_{0}^1 e_{1},
\end{equation*}
where $e_{1}$ is the first vector for the canonical basis and $X_{0}^1$ is a square integrable real-valued random variable. 
Because of the isotropy of the original equation, this choice is not restrictive.
To make things simpler, additionally to the centering assumption, recall  $\E\bigl[X_0^1\bigr]=0 $, we will also suppose 
(without any loss of generality)
that $X_{0}^1$ is  reduced, that is
\begin{equation*}
\E \bigl[ \bigl( X_{0}^1 \bigr)^2 \bigr] = 1.  
\end{equation*}
Given a real $x^1_{0}$, 
we will work under the conditional measure given $\{X^1_{0}=x_0^1 \}$, which we will still denote
by $\P$. Therefore, recalling \eqref{lan4} and \eqref{lan7},
we will write in the whole section
$(X_{t})_{t \geq 0}$ as
\begin{equation}
\label{eq:23:11:2}
X_{t} =  Z_t(x_0^1 e_1)
 - Z_{t} \int_{0}^t  Z_{s}^{\top} d\bar B_s, 
\quad t \geq 0, 
\end{equation}
which is understood as the conditional version of the original process $(X_{t})_{t \geq 0}$ given the initial condition 
$X_{0}=x_0^1 e_1$.
In this framework, the typical scales of $X_{t}$ in small time $t$ are given by:
\begin{equation}
\label{eq:23:11:3}
\E \left[ \vert X_{t}^1 -(Z_t x_0^1 e_1)^1\vert^2 \right] \sim_{t \rightarrow 0} t^2, \quad \E \left[\vert X_{t}^i-(Z_t x_0^1 e_1)^i \vert^2 \right] \sim_{t \rightarrow 0} t, \quad 2 \leq i \leq N,
\end{equation}
showing that the fluctuations of the density is $t$ in the first component and $t^{1/2}$ and the other ones. 
Eq. \eqref{eq:23:11:3} will be proved below.

\subsection{Small time expansions}
The key point in the whole analysis lies in small time expansions of the process $(Z_{t})_{t \geq 0}$ and  of
 the `conditional covariance' matrix 
$C_{t}$ in \eqref{eq:cove}. The precise strategy is to expand both of them in small times, taking care of the tails of the 
remainders in the expansion (recalling that the covariance matrix is random). 
We thus remind the reader of the so-called Bernstein equality, that will play a major role in the whole proof, see e.g. Revuz and Yor \cite{revu:yor:99}:
\begin{prop}
\label{BERNSTEIN}
Let $(M_{t})_{t \geq 0}$ be a continuous scalar martingale satisfying $M_{0}=0$. Then, for any $A>0$ and $\sigma>0$,
\begin{equation*}
\P \Bigl( 
M^*_{t}  \geq A, \ \langle M \rangle_{t} \leq \sigma^2 \Bigr) \leq 2\exp \bigl( - \frac{A^2}{2\sigma^2} \bigr),
\end{equation*}
where we have used the standard notation $M^*_{t}:=\sup_{0 \leq s \leq t} \vert M_{s} \vert$. 
\end{prop}
\begin{rmk}[Notation for supremums] With a slight abuse of notation,  for a 
process $(Y_t)_{t\ge 0}$ with values in $\R^\ell$, $\ell \geq 1$, 
we will denote $Y_t^*:=\max_{i \in \{1,\dots,\ell\}} (Y_t^{\ell})^*$.
Identifying $\R^{\ell}\otimes \R^k$ with $\R^{\ell \times k}$, we will also freely use those
notations for matrix valued processes.
\end{rmk}

\subsubsection{Landau notations revisited}
\label{subsub:landau}
In order to express the remainders in the expansion of the covariance matrix in a quite simple way,  
we will make a quite intensive use of Landau notation, but in various forms:   

\begin{defi}[Laudau notations]
\label{DEF_DET_REMAINDERS} 
Given some $T>0$, we let:

$(i)$ Given a deterministic function $(\psi_t)_{0 \leq t \leq T} $ (scalar, vector or matrix valued), we write 
$\psi_t={\mathcal O}(t^\alpha)$, for some $\alpha \ge 0$ and for any $t\in [0,T]$ if
there exists a constant $C:=C(N,T)$ such that $|\psi_t| \le Ct^{\alpha}$.

$(ii)$ Given a process $(\Psi_t)_{0 \leq t \leq T} $ (scalar, vector or matrix valued), we write 
$\Psi_t=O(t^\alpha)$, for some $\alpha \ge 0$ and for any $t\in [0,T]$ if
there exists a constant $C:=C(N,T)$ such that $|\Psi_t| \le Ct^{\alpha}$ a.s. Moreover,  
we write 
$\Psi_t=O_\P(t^\alpha)$, for some $\alpha \ge 0$ and for any $t\in [0,T]$ if,  
for all $p\in \N^*$, there exists a constant $C:=C(N,T,p)$ such that $\E[|\Psi_t|^p]^{1/p}\le Ct^{\alpha}$.
\end{defi}

\subsubsection{Small time expansion of the Brownian motion on ${\rm{SO}}_{N}(\R)$}
Following the proof of Lemma \ref{l4}, we then expand $Z_{t}$ according to 
$$Z_t = 
\exp\bigl(-(N-1)t\bigr)\left({\rm Id}_N+B_t+S_{t} \right) = \exp\bigl(-(N-1)t\bigr)\left({\rm Id}_N+B_t+\int_0^tdB_s B_s  +R_{t} \right),$$
for $t \geq 0$, with 
\begin{equation*}
S_{t} = \int_{0}^t dB_{s} \int_{0}^s dB_{r}\tilde{Z}_{r}, \quad R_{t} = \int_{0}^t dB_{s} \int_{0}^s dB_{r} 
\int_{0}^r dB_{u}\tilde{Z}_{u},
\quad \tilde{Z}_{t} = \exp\bigl(N-1)t\bigr) Z_{t}. 
\end{equation*}
Given some time horizon $T>0$,
the remainders $(S_{t})_{0 \leq t \leq T}$ and
$(R_{t})_{0 \leq t \leq T}$ can be controlled as follows on $[0,T]$:
\begin{lemma}
\label{lem:22:11:1} There exists $C:=C(N,T) >0$ such that, for all $t \in [0,T]$ and $y>0$, 
\begin{equation*}
\begin{split}
&\P \bigl(  S_{t}^* 
\geq y \bigr) \leq 2 \exp 
\bigl( -  \frac{\vert y \vert}{Ct} \bigr), \quad
 \P \bigl(   R_{t}^*  \geq y \bigr) \leq 2  \exp \bigl( -  \frac{\vert y \vert^{2/3}}{Ct} \bigr).
\end{split}
\end{equation*}
\end{lemma}
\begin{proof} 
Applying Bernstein's inequality
componentwise and using the fact that
$\| \tilde Z_r\|\le \exp((N-1)T)$ for $r \in [0,T]$, there exists a constant $C>0$ such that, for all $t \in [0,T]$,
\begin{equation}
\label{eq:25:12:7}
\P\biggl( \sup_{0 \leq s \leq t} \biggl\vert \int_{0}^s dB_{r} \tilde{Z}_{r} \biggr\vert \geq y_{1} \biggr)
\leq 2 \exp \left( - \frac{y_{1}^2}{Ct} \right),
\end{equation}
for any $y_{1}>0$.
By Bernstein's inequality again,
\begin{equation}
\label{eq:25:12:8}
\P\biggl( 
S_{t}^* 
\geq y_{2}, 
\int_{0}^t \biggl\vert \int_{0}^s dB_{r} \tilde Z_{r} \biggr\vert^2 ds \leq t y_{1}^2 \biggr) 
\leq 2 \exp \left( - \frac{y_{2}^2}{C t y_{1}^2} \right),
\end{equation}
for any $y_{2}>0$. 
Similarly,
\begin{equation}
\label{eq:25:12:9}
\P\biggl( 
R_{t }^* 
\geq y_{3}, 
\int_{0}^t \vert S_{s} \vert^2 ds \leq t y_{2}^2 \biggr) 
\leq 2 \exp \left( - \frac{y_{3}^2}{C t y_{2}^2} \right), 
\end{equation}
for $y_{3} >0$. Choosing $y_{2} = \vert y \vert$ and $y_{1} = \vert y \vert^{1/2}$, we complete the proof 
of the first inequality
by adding \eqref{eq:25:12:7} and \eqref{eq:25:12:8}.
Choosing $y_{3} = \vert y \vert$, $y_{2} = \vert y \vert^{2/3}$ and $y_{1} = \vert y \vert^{1/3}$, we complete the proof
of the second inequality by adding \eqref{eq:25:12:7}, \eqref{eq:25:12:8} and \eqref{eq:25:12:9}. 
\end{proof}

What really counts in the sequel is the first column $(Z_{t}^{\cdot,1})$ of the matrix $Z_{t}$. 
By antisymmetry of the matrix-valued process $(B_{t})_{t \geq 0}$, 
the entries of the column $(Z_{t}^{\cdot,1})$ write
\begin{equation}
\label{eq:2:12:1}
\begin{split}
Z_{t}^{1,1} &= \exp[-(N-1)t]
\biggl(
1 + \sum_{j=2}^N \int_{0}^t  dB^{1,j}_{s} B_{s}^{j,1} + R^{1,1}_{t} \biggr)
\\
&= 1 - \frac{N-1}{2} t - \frac{1}{2} \sum_{j=2}^N  (B^{j,1}_{t})^2 
+ O \bigl( t^2 + t \vert B_{t}^{.,1} \vert^2 + \vert R_{t} \vert \bigr),
\\
 Z_{t}^{i,1} 
 &= \exp[-(N-1)t]
\biggl(B_{t}^{i,1} + S_t^{i,1}
\biggr)
 = \bigl( 1 + {\mathcal O}(t) \bigr) \biggl( B_{t}^{i,1} + S_t^{i,1}
\biggr), \quad i \in \{2,\dots,N\}.
\end{split}
\end{equation}

\subsubsection{Expression of the covariance matrix}
By \eqref{eq:cove} and \eqref{barb2}, we know 
\begin{equation}
\label{eq:cove2}
C_{t} = \int_{0}^t Z_{t}Z_{s}^{\top} \frac{d}{ds} \langle \bar{B}\rangle_{s} \bigl(Z_{t}{Z}_{s}^{\top}
\bigr)^{\top}ds. 
\end{equation}
By \eqref{barb2} and \eqref{DEV_COV}, 
we have
\begin{equation}
\label{barb2b}
\frac{d}{ds}\langle \B \rangle_{s} = \Lambda_{s}
= \frac{1}{N} \left[ N-1  + \exp(-2Ns) \right] \textrm{Id}_{N} - \exp(-2Ns) e_{1} \otimes e_{1}.
\end{equation}
We then notice that 
$C_{t}$ reads 
\begin{equation*}
{C}_{t} = \int_{0}^t \bar Z_s \Lambda_{t-s} \bar Z_{s}^\top ds,
\end{equation*}
where we have denoted $\bar Z_s:=Z_tZ_{t-s}^{\top},\ s\in [0,t] $.
 By the invariance in law of Lemma \ref{lem:27:12:1}, we know that 
 $(Z_{s})_{0 \leq s \leq t}$ and $(\bar Z_{s})_{0 \leq s \leq t}$ have the same law. 
 In particular, noting that $Z_t=\bar Z_t $, the following identity in law holds:
 \begin{equation}
\label{ID_LAW_MIN}
 (Z_t,C_t)\overset{({\rm law})}{=}(Z_t,\bar C_t),
 \end{equation}
where 
 \begin{equation}
\label{eq:27:12:6:bis}
\bar C_t:=\int_{0}^t Z_{s} \Lambda_{t-s} Z_{s}^{\top} ds.
\end{equation}
Proposition \ref{exp} thus yields:
\begin{equation}
\label{REP_DENS_DEG_APRES_ID_EN_LOI}
f_{x_0}(t,v) 
= \E \biggl[ (2 \pi)^{-N/2} {\rm det}^{-1/2} ( \bar C_t)
\exp \left( - \frac{1}{2} \bigl\langle v - Z_{t} x_{0}, \bar C_t^{-1} \bigl(v - Z_{t} x_{0} \bigr) \bigr\rangle \right) 
\biggr].
\end{equation}
Now, by \eqref{barb2b} and \eqref{eq:27:12:6:bis}, we can expand $\bar C_{t}$ into
 \begin{equation}
\label{eq:27:12:6}
\bar C_t:=\int_{0}^t Z_{s}\bigl( (1 - 2(t-s)) \textrm{Id}_{N} - (1 - 2N(t-s)) e_{1} \otimes e_{1} \bigr) Z_{s}^{\top} ds + O(t^3).
\end{equation}
%
%
%

\subsubsection{Expansion of the covariance matrix}
\label{subsubse:exp:cov}
We now expand the integrand that appears in 
\eqref{eq:27:12:6}.
\begin{equation}
\label{eq:26:12:5}
\begin{split}
&Z_{s} \bigl( (1 - 2(t-s)) \textrm{Id}_{N} - (1 - 2N(t-s)) e_{1}\otimes e_{1} \bigr) Z_{s}^{\top}\\
&= (1- 2(t-s)) \textrm{Id}_{N} - (1- 2 N(t-s)) (Z_{s}  e_{1}) \otimes (Z_{s} e_{1}) 
\\
&= (1- 2(t-s)) \textrm{Id}_{N} - (1- 2N(t-s)) Z_{s}^{\cdot,1} \otimes Z_{s}^{\cdot,1}.
\end{split}
\end{equation}
By \eqref{eq:2:12:1}, the entries 
of $(Z_s^{\cdot,1})\otimes (Z_s^{\cdot,1})=:{\mathcal T}_s$ write for all $s\in [0,T]$ 
and
 all $i,j \in \{2,\cdots,N\}$, 
\begin{equation}
\label{eq:22:11:5}
\begin{split}
{\mathcal T}_s^{1,1} &= 1-(N-1)s-\sum_{j=2}^N \bigl(B_s^{j,1}\bigr)^2
                +O\bigl(s^{2}  + \vert B_{s}^{.,1} \vert^4 + \vert R_{s} \vert + \vert R_{s} \vert^2 \bigr),
\\
{\mathcal T}_s^{1,i}&={\mathcal T}_s^{i,1}=
\bigl(1 +{\mathcal O}(s) \bigr) \bigl(B_s^{i,1}+ S_{s}^{i,1}\bigr) \bigl( 1 + S_{s}^{1,1} \bigr)
\\
&=B_s^{i,1}+O\bigl( \vert S_s^{-1,1}\vert + \vert S_{s} \vert^2 + \vert B_{s}^{i,1} \vert 
\vert S_{s} \vert 
+ s \vert B_{s}^{i,1} \vert 
\bigr),
\\
{\mathcal T}_s^{i,j}&={\mathcal T}_s^{j,i}=\bigl(1 +{\mathcal O}(s)\bigr)
\bigl(B_s^{i,1}+ S_{s}^{i,1} \bigr) \bigl(B_s^{j,1}+S_{s}^{j,1}
             \bigr)
             \\
             &=B_s^{i,1}B_s^{j,1}+O \bigl( 
             \vert B_{s}^{.,1} \vert \bigl\vert S_{s} \bigr\vert + \bigl\vert S_{s} \bigr\vert^2 + s \vert B_{s}^{.,1} \vert^2
             \bigr),
\end{split}
\end{equation}
where 
we have used the identity $Z_{s}^{1,1}=(1+ {\mathcal O}(s))(1+S_{s}^{1,1})$
in the second line and the notation 
$S_{s}^{-1,1}:=(0,S_{s}^{2,1},\dots,S_{s}^{N,1})$ in the third one.

Denoting by $\bar{{\mathcal T}}_{s}:=
(1- 2(t-s)) \textrm{Id}_{N} - (1- 2N(t-s)) {\mathcal T}_{s}$ the 
last term in \eqref{eq:26:12:5}, it be expanded as
\begin{equation*}
\begin{split}
&\bar{\mathcal T}_{s}^{1,1} = 2 (N-1) (t-s) + (N-1)s + \sum_{j=2}^N \bigl(B_s^{j,1}\bigr)^2
              +O\bigl(t^{2}  + (( B_{t}^{.,1})^*)^4 + R_{t}^* + (R_{t}^*)^2 \bigr),
\\
&\bar{\mathcal T}_{s}^{i,i} =  1 - (B_{s}^{i,1})^2 + 
 O\bigl( t+
            (B_t^{.,1})^*   S_{t}^*  + ( S_{t}^* )^2 + t ((B_{t}^{.,1})^* )^2
 \bigr),             
              \\
&\bar{\mathcal T}_{s}^{1,i} = \bar{\mathcal T}_{s}^{i,1} = - B_s^{i,1} +
O\bigl( t+ (S_t^{-1,1})^*+ ( S_{t}^* )^2 +   (B_{t}^{.,1})^*  S_{t}^*  
+ t (B_{t}^{.,1})^*  
\bigr),             
\\
&\bar{\mathcal T}_{s}^{i,j} = \bar{\mathcal T}_{s}^{j,i} = - B_s^{i,1}B_s^{j,1} +
O\bigl(  t+ (B_{t}^{.,1})^*   S_{t}^* + (S_{t}^*)^2 + t ((B_{t}^{.,1})^*)^2
 \bigr), \quad i  \not = j, 
\end{split}
\end{equation*}
 for $(i,j) \in \{2,\dots,N\}^2$. 
By \eqref{eq:27:12:6} and by integration, 
we thus derive the following expansions for the entries of $\bar C_t$: for all $(i,j) \in \{2,\dots,N\}^2$,
\begin{equation} 
\label{CTR_COV_ENTRIES}
\begin{split}
&(\bar C_t)^{1,1}=\int_0^t\sum_{j=2}^N (B_s^{j,1})^2ds +\frac32(N-1)t^2+
t O\bigl(t^{2} + ( (B_{t}^{.,1})^* )^4 + R_{t}^* + (R_{t}^*)^2 \bigr),
\\
&(\bar C_t)^{i,i}=t -\int_0^t(B_s^{i,1})^2ds +t O\bigl( t+
            (B_t^{.,1})^*   S_{t}^*  + ( S_{t}^* )^2 + t ((B_{t}^{.,1})^* )^2
 \bigr)
,
\\
&(\bar C_t)^{1,i} =(\bar C_t)^{i,1}=- \int_0^t B_s^{i,1}ds
+ t O\bigl( t+ (S_t^{-1,1})^*+ ( S_{t}^* )^2 +   (B_{t}^{.,1})^*  S_{t}^*  
+ t (B_{t}^{.,1})^*  
\bigr),
\\
&(\bar C_t)^{i,j}=(\bar C_t)^{j,i}=-\int_0^tB_s^{i,1}B_s^{j,1}ds +
t O\bigl(  t+ (B_{t}^{.,1})^*   S_{t}^* + (S_{t}^*)^2 + t ((B_{t}^{.,1})^*)^2
 \bigr), \quad i  \not = j. 
\end{split}
\end{equation}
By \eqref{eq:23:11:2}, Eq. 
\eqref{eq:23:11:3} follows from the bounds for $(\bar{C})_{1,1}$ and $(\bar{C}_{i,i})_{2 \leq i \leq N}$.

\subsection{Proof of the Lower Bound in Theorem \ref{thm:bd:degenerate}}
\label{SD_EST}
We start  from the representation formula \eqref{REP_DENS_DEG_APRES_ID_EN_LOI} derived from the identity in law \eqref{ID_LAW_MIN}. We insist here that we choose some `untypical' events for the Brownian path on ${\mathcal A}_N(\R) $ to derive the bounds of Theorem \ref{thm:bd:degenerate}.

\subsubsection{First Step}
The point is to find some relevant scenarios to explain the typical behavior of 
$f_{x_0}(t,v)$ in \eqref{REP_DENS_DEG_APRES_ID_EN_LOI}
. Given $\xi \in  (0,1]$ 
such that $ t /\xi^2 \leq 1$ and $\gamma \in (0,1]$, we thus introduce the events
\begin{equation}
\label{eq:26:12:1}
\begin{split}
&{\mathcal B}^{1} = \bigcap_{j=2}^N \left\{ \sup_{0 \leq s \leq t} \vert B^{j,1}_{s} - \frac{s}{t} \xi \vert \leq 
\gamma
\frac{t}{\xi} \right\}, \quad
{\mathcal B}^{i,j} = 
\biggl\{ 
\sup_{0 \leq s \leq t} \biggl\vert  
 \int_{0}^s  dB^{i,j}_{r} B_{r}^{j,1} \biggr\vert \leq t
\biggr\}, 
\end{split}
\end{equation}
for  $i,j \in \{2,\dots,N\}$.
We then let
\begin{equation}
\label{eq:26:12:2}
{\mathcal B} = {\mathcal B}^1 \cap \bigcap_{i,j=2}^N {\mathcal B}^{i,j}. 
\end{equation}
\begin{lemma}
\label{lem:27:12:5}
There exists a constant $c>0$ such that 
\begin{equation*}
\P \left( {\mathcal B}\right) \geq c \gamma^{(N-1)/2}
\Bigl(  1 \wedge \bigl( \frac{t^{1/2}}{\xi} \bigr) \Bigr)^{N(N-1)/2} \exp \Bigl(
 - (N-1) \frac{\xi^2}{t} \Bigr).
\end{equation*}
\end{lemma}
\begin{proof}
On the event ${\mathcal B}^1$, it holds, for all $j \in \{2,\dots,N\}$,
\begin{equation*}
(B_{t}^{j,1})^* = \sup_{0 \leq s \leq t} \vert B_{s}^{j,1} \vert \leq  
\Bigl( \xi + \gamma \frac{t}{\xi} 
\Bigr) \leq 2 \xi,
\end{equation*}
since we have $\gamma t/ \xi^2 \leq 1$.

 By independence of $B^{j,i}$ and $B^{k,1}$ for $i,j,k \in \{2,\dots,N\}$, we also know that, conditionally on ${\mathcal B}^1$, 
 the process 
$ 
  (\int_{0}^s  dB_{r}^{i,j} B_{r}^{j,1})_{0 \leq s \leq t}$
 behaves as a Wiener integral, with
a variance process less than $(4 \xi^2 s)_{0 \leq s \leq t}$. Therefore,  using a Brownian change of time, we obtain
 \begin{equation*}
\begin{split}
\P \left( {\mathcal B}^{i,j} \vert {\mathcal B}^1 \right) \geq 
\P \bigl( \sup_{0 \leq s \leq 4 \xi^2 t} \vert \beta_{s} \vert \leq t
\bigr),
\end{split}
\end{equation*}
where $(\beta_{s})_{s \geq 0}$ is a 1D Brownian motion. We deduce that there exists a constant $c>0$ (which value is allowed to increase from line to line) such that 
\begin{equation*}
\P \left( {\mathcal B}^{i,j} \vert {\mathcal B}^1 \right) \geq
 c \Bigl( 1 \wedge \bigl( \frac{t^{1/2}}{\xi}\bigr) \Bigr). 
\end{equation*}
In fact, we must bound from below the conditional probability
$
\P (  \cap_{i,j=2}^N {\mathcal B}^{i,j}
\vert {\mathcal B}^1)$.
By antisymmetry of the matrix $B$ 
and conditional independence of the processes $(B^{i,j})_{2 \leq i < j \leq N}$, we deduce that 
\begin{equation*}
\P \biggl( \bigcap_{i,j=2}^N {\mathcal B}^{i,j} \, \vert {\mathcal B}^1 
\biggr) = \prod_{2 \leq i < j \leq N}
\P \left( {\mathcal B}^{i,j} \vert {\mathcal B}^1 \right) 
\geq c^{(N-1)(N-2)/2} \left( \min (1,\xi^{-1} t^{1/2}) \right)^{(N-1)(N-2)/2}. 
\end{equation*}
It thus remains to bound $\P({\mathcal B}^1)$ from below. 
For some $j \in \{2,\dots,N\}$, we deduce from Girsanov's theorem that 
\begin{equation*}
\begin{split}
&\P\biggl( \sup_{0 \leq s \leq t} \vert B^{j,1}_{s} - \frac{s}{t} \xi \vert \leq \frac{\gamma t}{\xi} \biggr)
=\E
\biggl[ \exp \Bigl( - 	\frac{\xi}{t}  \beta_{t} - \frac{\xi^2}{2 t}  \Bigr)
{\mathbf 1}_{\{\sup_{0 \leq s \leq t} \vert \beta_{s} \vert \leq \gamma  t/\xi\}}
\biggr]
\\
&\hspace{15pt}\geq \exp \Bigl( - 1 - \frac{\xi^2}{2 t}  \Bigr)
\P \Bigl( 
\sup_{0 \leq s \leq t} \vert \beta_{s} \vert \leq \frac{\gamma  t}\xi
\Bigr)
\geq c \gamma  \Bigl( 1 \wedge \bigl( \frac{t^{1/2}}{\xi}\bigr) \Bigr)
 \exp \Bigl( 
- \frac{\xi^2}{2t}
\Bigr), 
\end{split}
\end{equation*}
where $(\beta_{s})_{s \geq 0}$ is a 1D Brownian motion. 
By independence of the processes $(B^{1,j})_{2 \leq j \leq N}$, we deduce that  
\begin{equation*}
\P \left( {\mathcal B}^1 \right) \geq c^{N-1} \gamma^{N-1}
\Bigl( 1 \wedge \bigl( \frac{t^{1/2}}{\xi}\bigr) \Bigr)^{N-1} \exp \Bigl( - (N-1) \frac{\xi^2}{2t} \Bigr). 
\end{equation*}
We finally deduce that 
\begin{equation*}
\P \left( {\mathcal B} \right) \geq c^{N(N-1)/2} 
\gamma^{N-1}
 \exp \Bigl( - (N-1) \frac{\xi^2}{2t} \Bigr)
 \Bigl( 1 \wedge \bigl( \frac{t^{1/2}}{\xi}\bigr) \Bigr)^{N(N-1)/2}.
\end{equation*}
\end{proof}

\subsubsection{Second Step}
We now plug the analysis of the covariance matrix performed in \S \ref{subsubse:exp:cov} 
into the previous step: We compute the typical values of the conditional covariance matrix on the event ${\mathcal B} \cap {\mathcal R}$,  
where
\begin{equation}
\label{eq:26:12:3}
{\mathcal R} =
\bigl\{ S_{t}^* \leq \xi^{3/2} \bigr\} \cap
 \bigl\{ R_{t}^* \leq \xi^{9/4} \bigr\},
\end{equation}
so that, by Lemma
\ref{lem:22:11:1},
$\P ({\mathcal R}^{\complement} ) \leq c^{-1} \exp (- c \xi^{3/2}/t)$.
Therefore,
\begin{equation*}
\begin{split}
\P \left( {\mathcal B} \cap {\mathcal R} \right) &\geq 
c^{N(N-1)/2} \gamma^{N-1}
\exp \bigl( - (N-1) \frac{\xi^2}{2t} \bigr) 
\Bigl( 1 \wedge \bigl( \frac{t^{1/2}}{\xi}\bigr) \Bigr)^{N(N-1)/2}
- 
c^{-1} \exp \bigl(- c \frac{ \xi^{3/2}}{t} \bigr),
\end{split}
\end{equation*}
which proves that there exists a constant $C:=C(N) \geq 1$ (which value  is allowed to increase from line to line) such that 
\begin{equation*}
\P \left( {\mathcal B} \cap {\mathcal R} \right) \geq 
C^{-1}\gamma^{N-1}
\exp \bigl( - 2(N-1) \frac{\xi^2}{2t} \bigr) - 
c^{-1} \exp \bigl(- c \frac{ \xi^{3/2}}{t} \bigr),
\end{equation*}
using the fact that $1 \vee (\xi/t^{1/2}) \leq C \exp [\xi^2/(Nt) ]$. Therefore, for $\xi$ small enough, 
\begin{equation}
\label{eq:24:11:10}
\P \left( {\mathcal B} \cap {\mathcal R} \right) \geq 
C^{-1}
\exp \bigl( - 2(N-1) \frac{\xi^2}{2t} \bigr).
\end{equation}

On ${\mathcal B} \cap {\mathcal R}$ (see \eqref{eq:26:12:1} and \eqref{eq:26:12:2} for the definitions of ${\mathcal B}$ and 
\eqref{eq:26:12:3} for the definition of ${\mathcal R}$), we 
have 
\begin{equation}
\label{RANGE_REMAINDER}
S_{t}^* \leq \xi^{3/2}, \ R_{t}^* \leq \xi^{9/4}, \ (S_{t}^{-1,1})^* \leq t, \
(B_{t}^{.,1})^* \leq 2\xi, \ B^{1,i}_{s} B^{1,j}_{s} = (s/t)^2 \xi^2 + O(\gamma t),
\end{equation} 
the last expansion holding true 
for all $i,j \in \{2,\dots,N\}$ and following from the fact that $\gamma^2 t^2/\xi^2 \leq 
\gamma t$. 
We deduce from \eqref{CTR_COV_ENTRIES} that, for all $i,j \in \{2,\dots,N\}$,
on ${\mathcal B} \cap {\mathcal R}$,
\begin{equation*} 
\begin{split}
&(\bar C_t)^{1,1}= t (N-1) \frac{\xi^2}{3} +  O\bigl(t^{2} + t \xi^{9/4} \bigr),
\\
&(\bar C_t)^{i,i}=t \bigl( 1 - \frac{\xi^2}{3} \bigr) +
 O\bigl(t^{2} + t \xi^{9/4} \bigr),
\\
&(\bar C_t)^{1,i} =(\bar C_t)^{i,1}=- t \frac{\xi}{2}
+  O\bigl(\gamma t \xi + t^{2} + t \xi^{9/4} \bigr),
\\
&(\bar C_t)^{i,j}=(\bar C_t)^{j,i}=-t \frac{\xi^2}{3} 
+  O\bigl( t^{2} + t \xi^{9/4} \bigr), \quad i  \not = j,
\end{split}
\end{equation*}
the 
$O(\gamma t)$ in the third expansion following from the fact that 
$\gamma t^2/\xi = \gamma t \xi (t /\xi^2) \leq \gamma t \xi$.
Therefore, 
we can write, on ${\mathcal B} \cap {\mathcal R}$,
\begin{equation}
\label{eq:26:12:10}
\begin{split}
&\bar C_{t} =  \bar{C}_{t}^{0}
+ O \left( t^2+ t \xi^{9/4} \right),
\\
&\bar{C}_{t}^{0}
= 
 t \ \textrm{diag}(\xi,1,\dots,1)  \bar{C}^{00} \textrm{diag}(\xi,1,\dots,1),
\end{split}
\end{equation}
where $\textrm{diag}(\xi,1,\dots,1)$
denotes the diagonal matrix with $(\xi,1,\dots,1)$ as diagonal and where, for every $i \in \{2,\dots,N\}$,
\begin{equation*} 
\begin{split}
&(\bar C^{00})^{1,1}=   \frac{N-1}{3},
\quad (\bar C_t^{00})^{i,i}= \bigl( 1 - \frac{\xi^2}{3} \bigr),
\\
&(\bar C^{00})^{1,i} =(\bar C^{00})^{i,1}=- \frac{1}{2} + {\mathcal O}(\gamma), \quad (\bar C^{00})^{i,j}=(\bar C^{00})^{j,i}=- 
\frac{\xi^2}{3}, \quad i  \not = j. 
\end{split}
\end{equation*}

\subsubsection{Third Step}
We go thoroughly into the analysis of $\bar{C}_{t}^{0}$.
When $\xi=\gamma=0$, the determinant of $\bar{C}^{00}$ can be computed explicitly by adding $1/2$ times the column $i$ to the first column, for any $i =2, \dots,N$.  We obtain as a result 
\begin{equation*}
\left[ {\rm det}(\bar{C}^{00}) \right]_{\vert \xi = \gamma =0} = \frac{N-1}{12}.  
\end{equation*}
We deduce that 
\begin{equation}
\label{eq:27:12:15}
{\rm det}(\bar{C}_{t}^0) = t^N \xi^2 \left[ \frac{N-1}{12} + {\mathcal O}\left(\gamma + \xi^2 \right) \right]. 
\end{equation}
In a similar way, 
\begin{equation}
\label{eq:27:12:1}
\left( \bar{C}_{t}^0 \right)^{-1}= t^{-1} \textrm{diag}(1/\xi,1,\dots,1) \left( \bar{C}^{00} \right)^{-1}\textrm{diag}(1/\xi,1,\dots,1),
\end{equation}
where, for $\gamma$ and $\xi^2$ small enough,
\begin{equation*}
\left( \bar{C}^{00} \right)^{-1} = \left[\left( \bar{C}^{00} \right)^{-1}\right]_{\vert \xi = \gamma = 0} \left( \textrm{Id}_{N} +  {\mathcal O}(\gamma + \xi^2) \right),
\end{equation*}
with
\begin{equation}
\label{eq:28:12:1}
\left[\left( \bar{C}^{00} \right)^{-1}\right]_{\vert \xi = \gamma = 0} = {\mathcal O}(1),
\end{equation}
so that, by \eqref{eq:27:12:1},
$( \bar{C}_{t}^0 )^{-1}= {\mathcal O} ( t^{-1} \xi^{-2} )$. Therefore, referring to \eqref{eq:26:12:10}, we write
\begin{equation}
\label{eq:27:12:16}
\bar C_{t} = \bar{C}_{t}^0 + M_{t},
\end{equation}
with $M_{t} = O(t^{2} + t \xi^{9/4})$ on ${\mathcal B} \cap {\mathcal R}$, and we let
\begin{equation*}
\textrm{Id}_{N} + M_{t}' := \left( \bar{C}_{t}^{0}\right)^{1/2} \left( \bar{C}_{t}^0 + M_{t} \right)^{-1} 
\left( \bar{C}_{t}^{0}\right)^{1/2},  
\end{equation*}
where the exponent $1/2$ indicates the symmetric square root. Indeed, when $\gamma=\xi=0$, $\bar{C}^{00}$ is the  covariance matrix
of the vector 
\begin{equation*}
\biggl( \left( \frac{N-1}{12} \right)^{1/2} \zeta_{1} - \sum_{i=2}^N \frac{\zeta_{i}}{2} ,  \zeta_{2},\dots,
\zeta_{N} \biggr),
\end{equation*}
with $(\zeta_{1},\dots,\zeta_{N}) \overset{({\rm law})}{=} {\mathcal N}^{\otimes N}(0,1)$, so that it is non-negative symmetric matrix; since its determinant is positive, it is a positive symmetric matrix. By continuity, this remains true for $\xi$ and $\gamma$ small enough. For the same values of $\xi$ and $\gamma$, \eqref{eq:26:12:10} says that $\bar{C}_{t}^0$ is also symmetric and positive. Then,
\begin{equation}
\label{eq:23:11:10}
\left( \bar{C}_{t}^0 + M_{t} \right)^{-1} =  \left( \bar{C}_{t}^0 \right)^{-1/2}
\left( \textrm{Id}_{N} + M_{t}' \right) \left( \bar{C}_t^{0} \right)^{-1/2}. 
\end{equation}
As $M_{t}( \bar{C}_{t}^0 )^{-1} = O(t/\xi^2 + \xi^{1/4})$ is small when $t/\xi^2$ and $\xi$ are small, we 
can write 
\begin{equation*}
\begin{split}
\textrm{Id}_{N} + M_{t}' &= \left( \bar{C}_{t}^0 \right)^{1/2} \left( \bar{C}_{t}^0 \right)^{-1}\left[ \textrm{Id}_{N} + M_{t} \left(\bar{C}_{t}^0\right)^{-1} \right]^{-1} \left( \bar{C}_{t}^0 \right)^{1/2}
\\
&=  \left( \bar{C}_{t}^0 \right)^{-1/2} \sum_{n \geq 0} \left[ - M_{t} \left( \bar{C}_{t}^0 \right)^{-1} \right]^n 
\left( \bar{C}_{t}^0 \right)^{1/2}
\\
&= \textrm{Id}_{N} + \sum_{n \geq 0} (-1)^{n+1} \left(\bar{C}_{t}^0\right)^{-1/2} \left[ M_{t}  \left( \bar{C}_{t}^0 \right)^{-1} 
\right]^{n}
M_{t} \left( \bar{C}_{t}^0 \right)^{-1/2}.
\end{split} 
\end{equation*}
By \eqref{eq:27:12:1}, $( \bar{C}_{t}^0)^{-1/2} = O (t^{-1/2}\xi^{-1})$, so that 
$M_{t} ( \bar{C}_{t}^0)^{-1/2}
= O(t^{3/2} \xi^{-1} + t^{1/2} \xi^{5/4})$. Therefore,
\begin{equation*}
\begin{split}
\sum_{n \geq 0} (-1)^{n+1} \left(\bar{C}_{t}^0\right)^{-1/2} \left[ M_{t}  \left( \bar{C}_{t}^0 \right)^{-1} \right]^{
n}
M_{t} \left( \bar{C}_{t}^0 \right)^{-1/2}
&= \sum_{n \geq 0} \left[ O(t/\xi^2 + \xi^{1/4}) \right]^{n+1} 
\\
&= O(t/\xi^2 + \xi^{1/4}),
\end{split}
\end{equation*}
provided $t/\xi^2$ and $\xi$ are small enough. 

Therefore, for $t/\xi^2$ and $\xi$ small enough, the matrix $\textrm{Id}_{N} + M_{t}'$, which is symmetric by construction, has all its eigenvalues between $1/2$ and $2$, so that, for given a vector $v=(v_{1},\dots,v_{N})^{\top}$,
\eqref{eq:23:11:10} yields
\begin{equation}
\label{eq:27:12:10}
\begin{split}
\langle v , \left( \bar{C}_{t}^0 + M_{t} \right)^{-1} v \rangle 
&\leq C \big\langle \bigl(v_{1},v_{2},\dots,v_{N} \bigr)^{\top}, \bigl( \bar{C}_{t}^{0} \bigr)^{-1}
\bigl(v_{1},v_{2},\dots,v_{N} \bigr)^{\top} \big\rangle
\\
&\leq C t^{-1} \Big\langle \Bigl(\frac{v_{1}}{\xi},v_{2},\dots,v_{N} \Bigr)^{\top}, \bigl( 
\bar{C}_{t}^{00} \bigr)^{-1}
\Bigl(\frac{v_{1}}{\xi},v_{2},\dots,v_{N} \Bigr)^{\top} \Big\rangle
\\
&\leq C t^{-1} \biggl( \frac{v_{1}^2}{\xi^2} + \sum_{i=2}^N v_{i}^2 \biggr). 
\end{split}
\end{equation}

\subsubsection{Final Step}
\label{subsubsec:final}
We can summarize what we have proved in the following way: There exists a constant $K:=K(N) \geq 1$ such that, for 
$\max(t/\xi^2,\xi^2,\gamma) \leq 1/K$, Eq. \eqref{eq:27:12:10} holds
for any $(v_{1},\dots,v_{N}) \in \R^N$ on the event ${\mathcal B} \cap {\mathcal R}$. 

The point is now to plug $(v^1 - Z_{t}^{1,1} x_{0}^1,v^2-Z_t^{2,1}x_0^1,\dots,v^N-Z_t^{N,1}x_0^1)$ instead of $(v^1,\dots,v^N)$ in 
\eqref{eq:27:12:10}. Put it differently, we are to bound:
\begin{equation}
\label{EQ_MINO_2}
\inf_{Kt\le \xi^2\le 1/K} I(t,x_0,v,\xi),\quad I(t,x_0,v,\xi):=\Bigl[ \frac{\vert v^1 - Z_t^{1,1} x_{0}^1\vert^2}{\xi^2}+\sum_{i=2}^N \vert v^i-Z_t^{i,1}x_0^1\vert^2
\Bigr].
\end{equation}
By \eqref{eq:2:12:1} and \eqref{RANGE_REMAINDER}, on ${\mathcal B} \cap {\mathcal R}$, 
\begin{equation*}
\begin{split}
&Z_{t}^{1,1} = 1 + O \bigl( t + \xi^{2} \bigr) =  1 + O(\xi^2), 
\\
&Z_t^{i,1}= \bigl(1+{\mathcal O} (t) \bigr) \bigl( \frac{\gamma t}\xi+\xi + t\bigr) = O(\xi),
\end{split}
\end{equation*}
where we have used $t \leq \xi^2$ in both expansions. Pay attention that this step is crucial
as, together with the previous paragraph, it gives the joint behavior of $(Z^{\cdot,1}_{t},\bar{C}_{t})$
on ${\mathcal B} \cap {\mathcal R}$.

Therefore, we can find a constant $C:=C(N) >0$ such that   
\begin{equation}
\label{EQ_MINO_1}
\begin{split}
&\bigl\vert v^1 - Z_{t}^{1,1} x_{0}^1 \bigr\vert \leq 
\bigl\vert v^1 - x_0^1 \bigr\vert + C \xi^2 \vert x_{0}^1 \vert,
\\
&\bigl\vert v^i - Z_{t}^{i,1} x_{0}^1 \bigr\vert \le \vert v^i\vert + C \xi |x_0^1| , \quad i \not = 1.
\end{split}
\end{equation}
The value of $C$ being allowed to increase from line to line, 
we get: 
\begin{equation*}
I(t,x_0,v,\xi)\le C\biggl[ \frac{\vert v^1 - x_0^1\vert^2}{\xi^2}+|x_0^1|^2\xi^2+\sum_{i=2}^N|v^i|^2\biggr].
\end{equation*}
We now handle the minimization problem in \eqref{EQ_MINO_2} according to the value of
\begin{equation*}
\varsigma:= \frac{\vert v^1 - x_{0}^1 \vert}{1 \vee \vert x_{0}^1 \vert} 
\end{equation*}
If
$\varsigma \leq Kt$, we choose
$\xi^2 = Kt$ in the infimum. We obtain 
\begin{equation*}
\inf_{K t \leq \xi^2 \leq 1/K} I(t,x_0,v,\xi)
 \leq  C\biggl(2K t \bigl( 1 \vee \vert x_{0}^1 \vert^2 \bigr) +\sum_{i=2}^N|v^i|^2\biggr). 
\end{equation*}
If  
$\varsigma \geq 1/K$, we choose
$\xi^2 = 1/K$ in the infimum. We obtain 
\begin{equation*}
\inf_{K t \leq \xi^2 \leq 1/K} I(t,x_0,v,\xi) \leq  C
\biggl( 2K \vert v^1 - x_{0}^1 \vert^2+\sum_{i=2}^N|v^i|^2\biggr).
\end{equation*}
If  
$\varsigma \in [Kt, 1/K]$, we choose
$\xi^2 = 
\varsigma
$ in the infimum. We obtain 
\begin{equation*}
\inf_{K t \leq \xi^2 \leq 1/K} I(t,x_0,v,\xi)\leq  C \biggl( 2 \bigl(1 \vee \vert x_{0}^1 \vert \bigr) 
\vert v^1 - x_{0}^1 \vert + \sum_{i=2}^N|v^i|^2\biggr). 
\end{equation*}
This gives a lower bound for the exponential factor in \eqref{eq:27:12:10}
on the event ${\mathcal B} \cap {\mathcal R}$. 
When $x_0\in[-C_0,C_0]$, we can modify $C$ (allowing it to depend on $C_{0}$) in such a way that, in 
any of three cases,
\begin{equation}
\label{EQ_MINO_3}
\inf_{K t \leq \xi^2 \leq 1/K} I(t,x_0,v,\xi)\leq  C \biggl(  
\vert v^1 - x_{0}^1 \vert + \vert v^1 - x_{0}^1 \vert^2 + \sum_{i=2}^N|v^i|^2\biggr), 
\end{equation}
which fits the off-diagonal cost in the statement of Theorem \ref{thm:bd:degenerate}. Notice that the dependence 
of $C$ upon $C_{0}$ can be made explicit.

It remains to discuss the diagonal rate. By \eqref{eq:27:12:15} and \eqref{eq:23:11:10},
on ${\mathcal B} \cap {\mathcal R}$,
\begin{equation*}
{\rm det}(\bar{C}_{t})  = {\rm det}(\bar{C}_{t}^0 + M) \leq C' t^N \xi^2 = C' t^{N+1} (\xi^2/t) ,
\end{equation*}
for some constant $C'$. Now, 
\begin{equation*}
\frac{\xi^2}{t} 
\left\{
\begin{array}{ll}
\displaystyle  = K &\quad \textrm{if} \ 
\varsigma \leq Kt,
\vspace{5pt}
\\
\displaystyle  = \frac{\varsigma}t
 \leq \exp \Bigl( \frac{\vert v^1 - x_{0}^1 \vert}{ t} \Bigr)
&\quad \textrm{if} \  
\varsigma \in [Kt,1/K],
\vspace{5pt}
\\
\displaystyle= \frac{1}{Kt} \leq \frac{\varsigma}{t}
\leq  \exp \Bigl(  \frac{\vert v^1 - x_{0}^1 \vert}{ t}
 \Bigr)
&\quad \textrm{if} \ 
 \varsigma \geq 1/K. 
\end{array}
\right.
\end{equation*}
Therefore, modifying $C'$ if necessary,
\begin{equation}
\label{eq:27:12:21}
\left[ {\rm det}(\bar C_{t}) \right]^{-1/2} \geq (C')^{-1/2} t^{-(N+1)/2} 
\exp \Bigl( - \frac{\vert v^1 - x_{0}^1 \vert}{ t}
 \Bigr).
\end{equation}
In the same way, \eqref{eq:24:11:10} implies
\begin{equation}
\label{eq:27:12:22}
\P \left( {\mathcal B} \cap {\mathcal R}\right) \geq 
(C')^{-1}
\exp \Bigl( - C' \frac{\vert v^1 - x_{0}^1 \vert}{ t}
\Bigr).
\end{equation}
By \eqref{REP_DENS_DEG_APRES_ID_EN_LOI}, 
\eqref{eq:27:12:10}, \eqref{EQ_MINO_3}, \eqref{eq:27:12:21} and \eqref{eq:27:12:22}, we 
complete the proof of the lower bound. Indeed, for 
$x_0\in[-C_0,C_0]$ and $C:=C(N,C_{0})$,
\begin{equation*}
f_{x_0}(t,v) 
\geq \frac{1}{C t^{(N+1)/2}} \exp\biggl(- C \biggl[ 
\frac{\vert v^1 - x_{0}^1 \vert}{t}+ 
\frac{ \vert v^1 - x_{0}^1 \vert^2}{t} 
+ \sum_{i=2}^{N}\frac{|v^i|^2}t\biggr] \biggr).
\end{equation*}
\color{black}
%
%
%

\subsection{Proof of the Upper Bound in Theorem \ref{thm:bd:degenerate}}
Let us restart from the expression of the conditional density given by Proposition \ref{exp} that we recall here. For all $(t,x_0,v)\in \R^{+*}\times (\R^N)^2 $ we have:
\begin{equation*}
f_{x_{0}}(t,v)=\E\left[ \frac{1}{(2\pi)^{N/2}{\rm det}(C_t)^{1/2}}\exp\left(-\tfrac 12\langle C_t^{-1}(v-Z_tx_0),v-Z_tx_0\rangle \right)\right]. 
\end{equation*}

In order to handle the degeneracy in the first coordinate, we introduce the \textit{rescaled} covariance matrix
(pay attention that the notation $M$ below has nothing to do with the one used in
\eqref{eq:27:12:16})
\begin{equation*}
{M}_{t} := t^{-1} \T_{t}^{-1} {C}_{t} \T_{t}^{-1},  
\end{equation*}
where $\T_t$ is the $N \times N$-diagonal matrix:
\begin{equation}
\label{PROC_SCALED}
\T_t:= {\rm diag}(t^{1/2},1,\dots,1),
\end{equation} 
the matrix $t^{1/2} \T_{t}$ expressing the different scales in the fluctuations of the system, as emphasized in 
\eqref{eq:23:11:3}. 
Writing 
${C}_{t} = t^{1/2} \T_{t} {M}_{t}(t^{1/2} \T_{t})$
in the definition of $f_{x_0}(t,v)$, we deduce that 
\begin{equation}
\label{EXPR_P_BAR}
\begin{split}
f_{x_{0}}(t,v)  
&=\E\biggl[ 
\frac{1}{(2\pi)^{N/2}{\rm det}( M_t)^{1/2}t^{(N+1)/2}}
\\
&\hspace{15pt} \times
\exp\left(-\tfrac 12 \bigl\langle  M_t^{-1} \bigl[t^{-1/2}\T_t^{-1}( v-Z_t x_0)\bigr],t^{-1/2}\T_t^{-1}(v-Z_tx_0)
\bigr\rangle \right)
\biggr].
\end{split}
\end{equation}
This representation makes the explosion rate of the density along the diagonal appear, provided the determinant of the
 matrix ${M}_{t}$ is well-controlled as $t$ tends to $0$. In order to get the off-diagonal decay of the density, 
 we have in mind to perform a Gaussian integration by parts, in its most direct version, in order to bound the density by the tails of 
 the marginal distributions of the process 
 \begin{equation*}
 Z_t\Gamma_t,\ \Gamma_t:=\int_0 ^t Z_s^{\top}d\bar B_s.
 \end{equation*}
 Such a strategy is inspired from the approach based on Malliavin calculus for estimating 
  densities,  see e.g. Kusuoka and Stroock \cite{Kusuoka:87}, but here we take benefit of the underlying Gaussian structure to make the integration by parts directly and thus avoid any further reference to Malliavin calculus. 

 \subsubsection{Main step}
We now establish the upper bound of Theorem \ref{thm:24:12:1} for $f_{x_{0}}(t,.)$ in \eqref{EXPR_P_BAR}.
Rewrite first
\begin{equation}
\label{EXPR_DENS_RESCA}
f_{x_{0}}(t,v)=\frac{1}{t^{(N+1)/2}}\E\bigl[ p_{t}\bigl(t^{-1/2}\T_t ^{-1}( v-Z_tx_0)\bigr)\bigr],
\end{equation}
where 
\begin{equation*}
p_{t}(y) :=  
\frac{1}{(2\pi)^{N/2}{\rm det}( M_t)^{1/2}}
\exp\left(-\tfrac 12 \langle  M_t^{-1} y,y\rangle \right), \quad y \in \R^N,
\end{equation*}
stands for the conditional density at time $t$ of $t^{-1/2}\T_t^{-1}Z_t\Gamma_t $ 
given the $\sigma$-field $\F_{t}^Z :=
\sigma( (Z_u)_{0 \leq u \leq t})$ (pay attention that ${M}_{t}$ is random). Since $p_{t}$ is smooth, we directly have
\begin{equation} 
\begin{split}
 p_{t}(y)
 &=(-1)^N\int_{\prod_{i=1}^N \{\textrm{sign}(y_{i}) z_i>|y_i| \}}\partial_{z_1,\cdots, z_N} p(t,z)dz\\
&= 
\frac{(-1)^N}{(2\pi)^{N/2}{\rm det}( M_t)^{1/2}}
\int_{\prod_{i=1}^N \{\textrm{sign}(y_{i}) z_i>|y_i| \}} 
\partial_{z_1,\cdots,z_N} \Bigl\{
\exp\left(-\tfrac{1}{2}\langle  M_t^{-1}z,z\rangle \right) \Bigr\} dz.
\label{eq:ipp}
\end{split}
\end{equation}
Let now, for any $1 \leq i \leq N$,
\begin{eqnarray*}
P_t^i(z):=\left(\partial_{z_i,\cdots, z_1} \left\{\exp\left(-\tfrac{1}{2}\langle  M_t^{-1}z,z\rangle \right)\right\}\right)\exp\left(\tfrac{1}{2}\langle  M_t^{-1}z,z\rangle\right), \quad z \in \R^N,
\end{eqnarray*}
which is a polynomial of the variable $z$ with degree $i$. Similarly to the Hermite polynomials, it can be defined by induction 
\begin{equation} 
\label{DEF_PI}
\begin{split}
&\forall z\in \R^N,\ P_t^1 (z)=-( M_t^{-1}z)_1,\\
&\forall i\in \{2,\cdots,N\}, \forall z\in \R^N,\  P_t^i(z)=\partial_{z_i}P_{t}^{i-1}(z)- ( M_t^{-1}z)_iP_t^{i-1}(z).
\end{split}
 \end{equation} 
The highest order term in $P_t^i(z) $ writes $(-1)^i \prod_{j=1}^i ( M_t^{-1}z)_j$. Moreover, if $N$ is odd (resp. even), there are only contributions of odd (resp. even) degrees of $z$ in $P_t^N(z)$.
 
In particular, we can compute, for any $z \in \R^N$,
\begin{eqnarray*}
P_t^2(z)&=&\prod_{i=1}^{2}( M_t^{-1}z)_i-( M_t^{-1})_{1,2},\\
P_t^3(z)&=&-\prod_{i=1}^3( M_t^{-1}z)_i+
\sum_{\sigma \in {\mathfrak{S}}_{3}}
(  M_t^{-1})_{\sigma(1),\sigma(2)}( M_t^{-1}z)_{\sigma(3)},
\end{eqnarray*}
where ${\mathfrak{S}}_{3}$ is the symmetric group on $\{1,2,3\}$. 
More generally, for any $1 \leq i \leq N$, we can find a polynomial function 
${\mathscr P}^i$ on $\R^i \times \R^{i(i-1)/2}$ such that 
\begin{eqnarray*}
P_t^i(z)&={\mathscr P}^i \Bigl( \bigl(({M}_{t}^{-1} z)_{j}\bigr)_{1 \leq j \leq i},
\bigl(({M}_{t}^{-1})_{j,k} \bigr)_{1 \leq j < k \leq i} \Bigr). 
\end{eqnarray*}
The family ${\mathscr P}^1, \dots, {\mathscr P}^N$ can be defined by induction by means of 
\eqref{DEF_PI}:
\begin{equation*}
\begin{split}
{\mathscr P}^{i}\bigl( (\zeta_{j})_{1 \leq j \leq i}, (\vartheta_{j,k})_{1 \leq j < k \leq i}
\bigr) &= \sum_{\ell=1}^{i-1} \vartheta_{\ell,i}
\partial_{\zeta_{\ell}}
{\mathscr P}^{i-1}\bigl( (\zeta_{j})_{1 \leq j \leq i}, (\vartheta_{j,k})_{1 \leq j < k \leq i-1}
\bigr) 
\\
&\hspace{15pt} - \zeta_{i} {\mathscr P}^{i-1}\bigl( (\zeta_{j})_{1 \leq j \leq i}, (\vartheta_{j,k})_{1 \leq j < k \leq i-1}
\bigr).
\end{split}
\end{equation*}

Denoting by $ M_t^{-1/2}$ the symmetric square root of $ M_t^{-1}$, we can express both ${M}^{-1}_{t} 
z$ and ${M}^{-1}_{t}$ in terms of quadratic combinations of 
${M}^{-1/2}_{t} 
z$ and ${M}^{-1/2}_{t}$. Therefore, we can find a polynomial function ${\mathscr Q}^N$ on $\R^N \times \R^{N^2}$
such that 
\begin{eqnarray*}
P_t^N(z)&={\mathscr Q}^N \Bigl( \bigl(({M}_{t}^{-1/2} z)_{j}\bigr)_{1 \leq j \leq N},
\bigl(({M}_{t}^{-1/2})_{j,k} \bigr)_{1 \leq j, k \leq N} \Bigr). 
\end{eqnarray*}
Then, 
\begin{equation*}
\begin{split}
&\partial_{z_1,\cdots,z_N} \Bigl\{
\exp\left(-\tfrac{1}{2}\langle  M_t^{-1}z,z\rangle \right) \Bigr\}
\\
&= P_t^N(z)\exp\left(-\tfrac{1}{2}\vert  M_t^{-1/2}z \vert^2 \right)
\\
&={\mathscr Q}^N \Bigl( \bigl(({M}_{t}^{-1/2} z)_{j}\bigr)_{1 \leq j \leq N},
\bigl(({M}_{t}^{-1/2})_{j,k} \bigr)_{1 \leq j, k \leq N} \Bigr)\exp\left(-\tfrac{1}{2}\vert  M_t^{-1/2}z \vert^2 \right), 
\end{split}
\end{equation*}
which permits to `absorb' the polynomial terms in $\bigl(({M}_{t}^{-1/2} z)_{j}\bigr)_{1 \leq j \leq N}$ inside the 
exponential. There exists a constant $c:=c(N)\in (0,1] $ such that 
\begin{equation} 
\label{BOUND_JT}
\left| \partial_{z_1,\cdots, z_N} \left\{\exp\left(-\tfrac{1}{2}\langle  M_t^{-1}z,z\rangle \right)\right\}\right|\le J_t(N) \exp\left(-c| M_t^{-1/2}z|^2 \right),
\end{equation}
where $J_{t}(N)$ reads
\begin{equation}
\label{BOUND_JTb}
J_{t}(N) :=
\Bigl\vert {\mathscr R}^N \Bigl(
\bigl(({M}_{t}^{-1/2})_{j,k} \bigr)_{1 \leq j, k \leq N} \Bigr) \Bigr\vert,
\end{equation}
for a polynomial function ${\mathscr R}^N$ on $\R^{N^2}$. 
Plugging \eqref{BOUND_JT} into \eqref{eq:ipp} we obtain:
\begin{eqnarray*}
 p_{t}(y)&\le& \frac{J_t(N)}{c^{N/2}} \int_{\prod_{i=1}^N \{\textrm{sign}(y_{i})z_i>|y_i| \}} \exp\left(
  - \frac{c}{2} | M_t^{-1/2}z|^2\right)\frac{c^{N/2} \det( M_t^{-1/2})dz}{(2\pi)^{N/2}}.
\end{eqnarray*}
The covariance matrix $M_{t}$ being given, 
the integral in the right-hand side can be interpreted as 
the probability that an $N$-dimensional centered Gaussian random vector 
with $c^{-1} M_{t}$ as covariance matrix be in the set $\{ z \in \R^N : 
\textrm{sign}(y_{i})z_i>|y_i|, \ i = 1,\dots,N \}.$
Conditionally on ${\mathcal F}_{t}^Z$, we know that $c^{-1/2} 
\T_{t}^{-1} Z_{t} \Gamma_{t}$ is precisely a centered Gaussian vector
with $c^{-1} M_{t}$ as covariance. Therefore,
choosing $y=t^{-1/2}\T_t^{-1}( v-Z_tx_0)$, we deduce from
\eqref{EXPR_DENS_RESCA}:
\begin{equation*}
\begin{split}
&f_{x_0}(t,v) \le \frac{1}{t^{(N+1)/2}c^{N/2}}
\\
&\hspace{15pt} \times
\E\biggl[J_t(N) \P\biggl( \bigcap_{i=1}^N \Bigl\{ 
\bigl\vert \bigl(t^{-1/2}\T_t^{-1}Z_t\Gamma_t\bigr)^i \bigr\vert >c^{1/2}\bigl | \bigl(t^{-1/2}\T_t^{-1}(v-Z_tx_0)\bigr)^i\bigr| \Bigr\} \big|{\F}_t^Z\biggr) \biggr].
\end{split}
\end{equation*}
Since the matrix $\T_{t}$ is diagonal, Cauchy-Schwarz inequality gives:
\begin{equation}
\label{MAJ_DENS_PRELIM}
\begin{split}
f_{x_0}(t,v)&\le \frac{\E[J_t(N)^2]^{1/2}}{t^{(N+1)/2}c^{N/2}}\E\biggl[\P \biggl(\bigcap_{i=1}^N \bigl\{ 
\vert (Z_t\Gamma_t)^i \vert>c^{1/2}\bigl | (v-Z_tx_0)^i\bigr| \bigr\}|{\F}_t^Z\biggr)^2\biggr]^{1/2}\\
        &\le \frac{\E[J_t(N)^2]^{1/2}}{t^{(N+1)/2}c^{N/2}}  
        \P \biggl( \bigcap_{i=1}^N \bigl\{ \vert (Z_t\Gamma_t)^i \vert>c^{1/2}\bigl | (v-Z_t x_0)^i\bigr| \bigr\}\biggr)^{1/2}.
\end{split}
\end{equation}
Formula \eqref{MAJ_DENS_PRELIM} is the starting point for the upper bound that follows from the next two Lemmas.

\begin{lemma}[Diagonal Controls]
\label{LEMME_COV}
Given $T>0$, there exists a constant $C:=C(N,T)$ such that, 
for all $t \in [0,T]$, $$\E[J_t(N)^2]^{1/2}\le C.$$
\end{lemma}

\begin{lemma}[Tail estimates]
\label{LEMME_QUEUES}
Given $T>0$,
there exists a constant $C:=C(N,T)\ge 1$ such that for all $t \in [0,T]$ and $v\in \R^N$:
\begin{equation*}
\begin{split}
&\P \biggl( \bigcap_{i=1}^N \bigl\{ \vert (Z_t\Gamma_t)^i \vert >c^{1/2}\bigl | ( v-Z_tx_0)^i\bigr| \bigr\} \biggr)
\\
&\hspace{15pt} \le C\E\biggl[\exp\biggl( - \frac{C^{-1}}t  \biggl\{ \vert ( v-Z_tx_0)^1 \vert
 + |( v-Z_tx_0)^1|^2 
+\sum_{i=2}^N \vert ( v-Z_tx_0)^i \vert^2\biggr\}
\biggr)\biggr]^{1/2}.
\end{split}
\end{equation*}
\end{lemma}
The proofs of Lemmas \ref{LEMME_COV} and \ref{LEMME_QUEUES} are given in the subsections \ref{SEC_PR_DIAG} and \ref{SEC_PR_QUEUES} respectively.

\subsubsection{Derivation of the diagonal controls}
\label{SEC_PR_DIAG}
This subsection is dedicated to the proof of Lemma \ref{LEMME_COV}. Usually, in the Malliavin calculus approach to density estimates, this step is the most involved and requires a precise control of the determinant of the Malliavin covariance matrix, see e.g. Kusuoka and Stroock \cite{kusu:stro:85} or Bally \cite{ball:90}. In the current framework the determinant of the `covariance' matrix $ M_t$ still plays a key role but the specific structure of that matrix, especially the fact that $(Z_s)_{s\ge 0}$ 
defines an isometry, yields the required estimate almost for free.

Precisely, we have the following Proposition.

\begin{prop}[Control of the determinant of the covariance]
\label{PROP_CTR_DET}
For a given $T>0$,
there exists $C:=C(N,T)$ such that, 
for all $t \in [0,T]$, almost surely,
$${\rm det}( M_t)^{-1}\le C.$$
\end{prop}
{\textit{Proof.}} Since $ M_t=t^{-1}\T_t^{-1}  C_t \T_t^{-1}$ and $\det(t^{1/2}\T_t)^{-2}=t^{-(N+1)}  $, we here concentrate on $\det( C_t)$. The claim of the proposition indeed follows from the bound
\begin{equation}
\label{DET_CT}
\det( C_t)\ge Ct^{N+1}.
\end{equation}
To derive \eqref{DET_CT} we recall the `variational' formulation of the determinant for symmetric matrices
(see for instance \cite{Delarue:Monge}).

\begin{lemma}[Variational expression of the determinant]
\label{LEMME_DET_VAR}
Let $A$ be a symmetric $N\times N $ matrix.  
Then 
$${\det}^{1/N}(A)=\frac{1}{N} \inf \{ {\rm Tr}(aA),\, a \in {\mathcal{S}}_N^+(\R),\, \det(a)=1\}, $$
where ${\mathcal S}_N^+(\R) $ stands for the set of symmetric positive $N\times N $ matrices. 
\end{lemma}

Recall the expression of $C_{t}$ from
\eqref{eq:cove2}.
Since $Z_t$ is an isometry, we have $\det(C_t)=\det(\hat C_t) $ where $\hat C_t=\int_0^t Z_s^\top d\langle \bar B\rangle _sZ_s$. From \eqref{barb2b}, we get 
that $\hat C_t=\int_0^t Z_s^\top D_s Z_s^{} ds$ with 
\begin{equation}
\label{eq:matrix:D}
D_s:= \textrm{diag}\Bigl(
(N-1)\frac{1- \exp(-2Ns)}{N},
1- \frac{1- \exp(-2Ns)}{N},\dots,
1- \frac{1- \exp(-2Ns)}{N}\Bigr),
\end{equation}
so that ${\rm det}(D_{s}) \geq Cs$ for $s \in [0,T]$ and for some constant $C:=C(N,T) >0$. Therefore, 
we derive from Lemma \ref{LEMME_DET_VAR} that,
for any $a \in {\mathcal S}_N^+$ with $\det(a)=1$,
\begin{equation*}
\begin{split}
\left\{ \frac 1N{\rm Tr}\left[ \left( \int_0^t Z_s^\top D_s Z_s^{}ds \right) a\right] \right\}^N
&=\left\{ \int_0^t \frac 1N{\rm Tr}\left(Z_s^\top D_s Z_s^{} a\right) ds \right\}^N
\\
&\ge \left\{ \int_0^t \det\left(Z_s^\top D_s Z_s^{}\right)^{1/N} ds \right\}^N 
\\
&\ge \left\{\int_0^t \det(D_s)^{1/N}ds\right\}^{N}\ge Ct^{N+1},
\end{split}
\end{equation*}
for a new value of $C$. 
Taking the infimum over $a$ and reapplying 
Lemma \ref{LEMME_DET_VAR}, this proves \eqref{DET_CT} and thus the proposition.
\hfill $\square $
\vspace{5pt}

To achieve the proof of Lemma \ref{LEMME_COV}, it therefore remains to check that the entries of the matrix $M_t$ are bounded in any $L^p(\P)$, $p \geq 1$ (uniformly on $[0,T]$). 
With the notation of Definition \ref{DEF_DET_REMAINDERS}, Lemma \ref{LEMME_COV} will follow from the control
\begin{equation}
\label{M_O_P1}
 \forall (i,j) \in \{1,\cdots, N \}^2,\ (M_t)_{i,j}=O_\P(1). 
\end{equation}
Associated with  Proposition \ref{PROP_CTR_DET}, this will indeed yield that 
$M_{t}^{-1/2}$ also satisfies \eqref{M_O_P1}
(by controlling from above and below the eigenvalues of ${M}_{t}$ in terms of its determinant and its norm).
Equation \eqref{M_O_P1} is easily derived from \eqref{ID_LAW_MIN}, \eqref{CTR_COV_ENTRIES} and the definition of the scale matrix $t^{1/2}\T_t$ in \eqref{PROC_SCALED}.\hfill $\square $

\subsubsection{Derivation of the tail estimates}
\label{SEC_PR_QUEUES}
This subsection is dedicated to the proof of Lemma \ref{LEMME_QUEUES}. Conditioning with respect to ${\mathcal F}_t^B:=\sigma((B_s)_{0 \leq s \leq t})$ (which is independent of $(\bar{B}_{s})_{s \geq 0}$),
\begin{equation*}
\begin{split}
\pi &:=\P \biggl( \bigcap_{i=1}^N \bigl\{ \vert (Z_t\Gamma_t)^i \vert >c^{1/2}\bigl | (v-Z_tx_0)^i\bigr| \bigr\} \biggr)
\\
&=\E\biggl[\P\biggl(\bigcap_{i=1}^N \biggl\{ \biggl\vert 
\biggl( Z_t\int_0^tZ_s^{\top} d\bar B_s\biggr)^i \biggr\vert >c^{1/2}\bigl | (v-Z_tx_0)^i\bigr| 	\biggr\} \big\vert \F_t^B\biggr)\biggr].
\end{split}
\end{equation*}

Since $(Z_s)_{0 \leq s \leq t}$ is an isometry, it is bounded and so is $(Z_tZ_s^\top)_{0 \leq s \leq t} $. 
Moreover, by 
\eqref{barb2b}, $d\langle \bar{B}_{t} \rangle/dt$ is less than $\textrm{Id}_{N}$
(in the sense of symmetric matrices). Therefore, 
By Proposition \ref{BERNSTEIN} (Bernstein inequality)
 applied to the conditionally Gaussian 
variables $ \bigr((\int_0^t Z_tZ_s^{\top}d\bar B_s)^i\bigl)_{i\in \{1,\cdots,N\}}$,
there exists a constant $\bar{c}:=\bar{c}(N) \geq 1$ such that
\begin{equation} 
\label{GAUSSIAN_TAIL}
\pi \le \bar c\E\biggl[\exp\biggl(- \frac{1}{\bar c}
\frac{|v^1-(Z_tx_0)^1|^2}{t}
- \frac{1}{\bar c} \sum_{i=2}^N \frac{| (v-Z_tx_0)^i |^2}{ t} \biggr)\biggr].
\end{equation}
Equation \eqref{GAUSSIAN_TAIL} provides us with the Gaussian part of the estimate. 
To derive the exponential one, we apply Chebychev inequality: for any $\gamma>0$,
\begin{equation}
\label{eq:21:11:1}
\begin{split}
\pi &\le \E\biggl[\exp\biggl(-\gamma c^{1/2}\frac{|v^1-(Z_tx_0)^1|}t \biggr)
\E\biggl[\exp\bigl(\frac \gamma t (Z_t\Gamma_{t})^1 \bigr)\I_{\cap_{i=1}^N \{ (Z_t\Gamma_t)^i>c^{1/2} | ( v-Z_tx_0)^i|\}}\bigr|\F_t^B \biggr] \biggr]\\
&\le \bar c^{1/2}\E\biggl[\exp\biggl(-\gamma c^{1/2} \frac{| v^1-(Z_tx_0)^1|}t 
- \frac{1}{2 \bar c} \frac{\vert (v^1 -(Z_t x_{0})^1 \vert^2}t 
-\frac{1}{2 \bar c}\sum_{i=2}^{N}\frac{|(v-Z_tx_0)^i|^2}t\biggr)
\\
&\hspace{15pt} \times
 \E\biggl[\exp\biggl(\frac {2\gamma} t \biggl( Z_t\int_0^t Z_s^{\top}d\bar  B_s \biggr)^1 \biggr)\biggr|\F_t^B\biggr]^{1/2}\biggr]\\
 &\le \bar c^{1/2} \E\biggl[\exp\biggl(-2\gamma c^{1/2} \frac{| v^1-(Z_tx_0)^1|}t 
-  \frac{1}{ \bar c} \frac{\vert (v^1 -(Z_t x_{0})^1 \vert^2}t 
- \frac{1}{\bar c} \sum_{i=2}^{N}\frac{|(v-Z_tx_0)^i|^2}t\biggr)\biggr]^{1/2}
\\
&\hspace{15pt} \times
 \E\biggl[\exp\biggl(\frac {2\gamma} t \biggl( Z_t\int_0^t Z_s^{\top}d\bar  B_s \biggr)^1 \biggr)\biggr]^{1/2}
 ,
\end{split}
\end{equation}
recalling $(Z_t\Gamma_t)^1=(Z_t\int_0^t Z_s^\top d\bar B_s)^1 $ and using also the Cauchy-Schwarz and Bernstein inequalities (similarly to  \eqref{GAUSSIAN_TAIL}) to pass from the first to the second line. Recalling 
\eqref{eq:cove2}
%
and using the
Gaussian character of the conditional distribution of  $ \int_0^t Z_s^{\top}d\bar B_s$ 
given $\F_t^B$, we get
\begin{equation}
\label{eq:21:11:2}
\begin{split}
\E\biggl[\exp\biggl(\frac {2\gamma} t \biggl( \int_0^t Z_t Z_s^{\top}d\bar  B_s \biggr)^1 \biggr)\biggr|\F_t^B\biggr]
= \E\biggl[\exp\bigl(\frac {2\gamma^2}{t^2} C_{t}^{1,1} \bigr)\bigr|\F_t^B\biggr]
\end{split}
\end{equation}
When taking the expectation, we know from the identity 
in law \eqref{ID_LAW_MIN} that we can replace $C_{t}$ by $\bar{C}_{t}$. 
By \eqref{eq:27:12:6}
and \eqref{eq:26:12:5}, 
\begin{equation*}
\begin{split}
\bar{C}_{t}^{1,1} &= \int_{0}^t \bigl( 1 - 2(t-s) - (1- 2N(t-s)) \bigl(Z_{s}^{1,1}\bigr)^2 \bigr) ds + O(t^3)
\\
&= \int_{0}^t \bigl( 1 -  \bigl(Z_{s}^{1,1}\bigr)^2 \bigr) ds + O(t^2)
= \int_{0}^t \bigl( 1 - Z_{s}^{1,1}\bigr)\bigl( 1 + Z_{s}^{1,1} \bigr) ds + O(t^2).
\end{split}
\end{equation*}
We then write $Z_{s}^{1,1} = (1+ {\mathcal O}(s))(1 + S_{s}^{1,1})$, which leads to a simplified version of
\eqref{CTR_COV_ENTRIES}:
\begin{equation*}
\bar{C}_{t}^{1,1} = t O \bigl( S_{t}^{*} \bigr)  + O(t^2).
\end{equation*}
The point is then to plug the above expansion into the expectation of \eqref{eq:21:11:2}. We thus compute the moments of the right-hand side above. We make use of  Lemma
\ref{lem:22:11:1}, which says that $S_{t}^*/t$ has an exponential tail. Therefore, choosing 
$\gamma$ small enough, we can bound 
the last factor in the right-hand side in
\eqref{eq:21:11:2}
by $\bar{C}:=\bar{C}(N,T)$. 
This completes the proof of Lemma \ref{LEMME_QUEUES}
$\square$

\subsubsection{Conclusion} Combining Lemmas \ref{LEMME_COV} and \ref{LEMME_QUEUES} we derive that, 
for $t \in [0,T]$,
\begin{equation}
\label{eq:cl:1}
\begin{split}
&f_{x_{0}}(t,v) 
\\
&\hspace{5pt} \leq \frac{C}{t^{(N+1)/2}}
\E\biggl[\exp\biggl(- \frac{1}{Ct} 
\biggl[ |v^1-( Z_{t} x_0)^1| + 
\vert v^1 -  (Z_{t} x_{0})^1 \vert^2 + 
\sum_{i=2}^{N} |v^i-(Z_tx_0)^i|^2 \biggr]\biggr) \biggr]^{1/2}, 
\end{split}
\end{equation}
with $C:=C(N,T)$.
Using Cauchy-Schwarz inequality, it suffices to bound 
\begin{equation*}
\begin{split}
&F_{1} := \E\biggl[\exp\biggl(- \frac{2}{Ct} 
\biggl[ |v^1-( Z_{t} x_0)^1| + 
\vert v^1 -  (Z_{t} x_{0})^1 \vert^2 \biggr]\biggr) \biggr]^{1/4},
\\
&F_{2} := \E\biggl[\exp\biggl(- \frac{2}{Ct}   
\sum_{i=2}^{N} |v^i-(Z_tx_0)^i|^2 \biggr) \biggr]^{1/4}.
\end{split}
\end{equation*}
We start with $F_{2}$. By the inequality,
$- 2|v^i-(Z_tx_0)^i|^2 \leq - |v^i|^2 +  2 \vert (Z_{t} x_{0})^i \vert^2$, we obtain 
\begin{equation}
\label{eq:cl:2}
\begin{split}
F_{2} &\leq \exp\biggl(- \frac{1}{4Ct}  
\sum_{i=2}^{N} |v^i|^2 \biggr) 
\E\biggl[\exp\biggl( \frac{2}{Ct} 
\sum_{i=2}^{N} |(Z_{t} x_0)^i|^2\biggr) \biggr]^{1/4}
\\
&= \exp\biggl(- \frac{1}{4Ct} 
\sum_{i=2}^{N} |v^i|^2 \biggr) 
\E\biggl[\exp\biggl( \frac{2\vert x_{0}^1\vert^2}{Ct} 
 \sum_{i=2}^{N} \bigl( Z_{t}^{i,1} \bigr)^2 \biggr) \biggr]^{1/4}.
\end{split}
\end{equation}
Now, 
$\sum_{i=2}^{N} ( Z_{t}^{i,1})^2 = 1 - (Z_{t}^{1,1})^2 = O( 1 - Z_{t}^{1,1}) = 
O(S_{t}^{1,1} + t)$.
Therefore, for $\vert x_{0}^1 \vert \leq C_{0}$, we deduce from Lemma
\ref{lem:22:11:1} that we can choose $C:=C(N,T,C_{0})$ large enough  
 in \eqref{eq:cl:1}
 such that the second factor in 
the last line is bounded by $C$.  

It remains to bound $F_{1}$. Given some $A>0$, we split the expectation according to the events 
$\{\vert Z_{t}^{1,1} - 1 \vert \vert x_{0}^1 \vert \leq A\}$ and 
$\{\vert Z_{t}^{1,1} - 1 \vert \vert x_{0}^1 \vert > A\}$. Using
the inequalities
$- 2|v^1-(Z_tx_0)^1|^2 \leq - |v^1 - x_{0}^1|^2 +  2 \vert (Z_{t}^{1,1}-1) x_{0}^1\vert^2$
and 
$- 2|v^1-(Z_tx_0)^1| \leq - |v^1 - x_{0}^1| +  2 \vert (Z_{t}^{1,1}-1) x_{0}^1\vert$, 
we obtain 
\begin{equation*}
\begin{split}
F_{1} &\leq \exp\biggl(- \frac{1}{4Ct} 
\biggl[ |v^1-x_{0}^1| + 
\vert v^1 -  x_{0}^1 \vert^2 \biggr]
+ \frac{1}{2Ct} \bigl( A+ A^2 \bigr)\biggr)
 + \P \bigl( 
\vert Z_{t}^{1,1} - 1 \vert \vert x_{0}^1 \vert > A
\bigr),
\end{split}
\end{equation*}
for $C:=C(N,T)$. 
Now, by Proposition \ref{BERNSTEIN} (Bernstein inequality), we have
\begin{equation*}
\P \bigl( \vert (Z_{t}^{1,1} - 1) x_{0}^1 \vert \geq A \bigr)
\leq 2 \exp \bigl( - \frac{A^2}{2(x_{0}^1)^2 t} \bigr).
\end{equation*}
On the other hand, since $Z^{1,1}_{t}- 1 = O(S_{t}^{1,1}+t)$, Lemma \ref{lem:22:11:1} yields
(for a possibly new value of $C$)
\begin{equation*}
\P \bigl( \vert (Z_{t}^{1,1} - 1) x_{0}^1 \vert \geq A \bigr)
\leq C \exp \bigl( - \frac{A}{C \vert x_{0}^1 \vert t} \bigr). 
\end{equation*}
Choosing $A= \vert x_{0}^1- v_{0}^1 \vert / 8$, we deduce that (once again, modifying $C$ if necessary)
\begin{equation}
\label{eq:cl:3}
\begin{split}
F_{1} &\leq C \exp\biggl(- \frac{1}{Ct} 
\biggl[ \frac{|v^1-x_{0}^1|}{1 \vee \vert x_{0}^1 \vert} + 
\frac{\vert v^1 -  x_{0}^1 \vert^2}{1 \vee \vert x_{0}^1 \vert^2} 
\biggr] \biggr).
\end{split}
\end{equation}
By \eqref{eq:cl:1}, \eqref{eq:cl:2} and \eqref{eq:cl:3}, 
we get an upper bound for $f_{x_{0}}(t,v)$. When $x_{0} \in [-C_{0},C_{0}]$, we can choose  
$C$, depending upon $N$, $T$ and $C_{0}$, in order to get the required estimate. (As in the lower 
bound, the dependence of $C$ upon $C_{0}$ can be made explicit.)  
\color{black}

\end{document}